\documentclass[]{article}
\usepackage[T1]{fontenc}
\usepackage[utf8]{inputenc}
\usepackage[english]{babel}
\usepackage{amsthm}
\usepackage{paralist}
\usepackage{amsfonts}
\usepackage{amsmath}
\usepackage{amssymb}
\usepackage{mathrsfs}
\usepackage{mathtools}
\usepackage{eufrak}
\usepackage{color}

\usepackage{authblk}

\numberwithin{equation}{section}

\newtheorem{theorem}{Theorem}[section]
\newtheorem{proposition}[theorem]{Proposition}
\newtheorem{lemma}[theorem]{Lemma}
\newtheorem{corollary}[theorem]{Corollary}

\theoremstyle{definition}
\newtheorem{remark}[theorem]{Remark}
\newtheorem{notation}[theorem]{Notation}
\newtheorem{definition}[theorem]{Definition}

\newcommand{\aff}{\text{aff}}
\newcommand{\interior}{\text{int}}
\newcommand{\relint}{\text{relint}}
\newcommand{\conv}{\text{conv}}
\newcommand{\epi}{\text{epi}}

\newcommand{\Pol}{\text{Pol}}

\title{Dimension-raising Homomorphisms between Lattices of Convex Bodies}
\author[1]{Luisa Allgaier\footnote{Work on this article was supported by the University of Mannheim through appointing the first author as a research assistant.}}
\author[1]{Heinz Weisshaupt\footnote{Main parts of this article have been written during a stay of the second author as visiting professor at the university of Mannheim.}}
\affil[1]{Department of Mathematics, University of Freiburg}

\begin{document}
\maketitle

\begin{abstract}
$ $\\
We settle the first unsolved case of a problem of P. M. Gruber, asked by him in 1991 in \cite{Gr91}, namely, to investigate the homomorphisms from the lattice of convex bodies of ${\mathbb{E}}^c$ to the lattice of convex bodies of ${\mathbb{E}}^d$ for $c<d$. We completely describe these homomorphisms for the case $d=c+1$, for $c \ge 3$. (The case $d=c$ was settled in \cite{Gr91}.) The obtained result is then applied to characterize anti-homomorphisms and homomorphisms from lattices of convex bodies to lattices of convex functions.\\ \\
MSC2010 52A20
\end{abstract}

\section{Introduction}
Endomorphisms of the lattice of convex bodies of ${\mathbb E}^{c}$ and the lattice of closed convex subsets of ${\mathbb E}^{c}$ have been characterized in \cite{Gr91} and \cite{SLO11} respectively. The connection between such endomorphisms and the concept of duality were investigated in \cite{BS} and \cite{SLO11}. Further characterisations regarding dualities, order reversing and order preserving maps on spaces of convex sets or functions have been obtained in \cite{AM07.1}, \cite{AM09.1}, \cite{AM09} and \cite{ASlo11}. The cited results have in common that the considered transformations admit representations that are either induced by affine bijections between underlying Euclidean spaces or are concatenations of affine bijections with special duality transformations. For further results concerning the representability of transformations on spaces of convex sets by affine bijections consult \cite[Theorem 13.4, 13.5, 13.7 and 13.8]{Gr07}, \cite{AM10} and \cite{MSeSlo11}. Note that in less rigid cases the class of affine bijections has to be replaced by the less restrictive class of projective ones (see \cite{We2001}, \cite{SeSlo} and \cite{AMF11}).\\ \\
However, the cited results---and the literature as far as known to us---also share the property, that the involved transformations have to preserve the dimension of convex sets or epigraphs of functions i.e. they are not dimension raising.\footnote{This is even true for \cite[Theorem 1]{AM10} in the sense that by the one-one correspondence between convex sets and 1-homogeneous functions \cite[Theorem 1]{AM10} is equivalent to the dimension preserving result \cite[Theorem 2]{AM10}, i.e. \cite[Theorem 1]{AM10}---seemingly concerning a dimension raising situation---reduces to \cite[Theorem 2]{AM10} and thus tacitly deals with a situation in that the dimension is preserved.} This is contrasted by our investigation that permits a raise of dimension by the homomorphisms and anti-homomorphisms under consideration.\\ \\
Note that the problem of determining the lattice homomorphisms 
from the lattice of compact convex sets of ${\mathbb{E}}^c$ to those of 
${\mathbb{E}}^d$, for $c<d$ was posed by P. M. Gruber in 1991. In \cite{Gr91} he writes:
\begin{enumerate}
\item[]
'It is an open problem to characterize the homomorphisms for $c<d$.`
\end{enumerate}
Utilizing second countability of ${\mathbb E}^{d}$---a surprisingly powerful tool already used in \cite{Gr91}---we obtain with Lemma \ref{Dimlemma} a general dimension bound for dimension raising homomorphisms for $c<d$. We then investigate the first interesting case of P.M. Grubers question, i.e., when 
$d=c+1$, and under the hypothesis $c \ge 3$ we completely solve this case 
of the problem.\\ \\
Beside our discovery/invention of Lemma \ref{Dimlemma} our proofs are based on completely new ideas utilizing a combinatorial theorem of Radon as well as transversality results related to a combinatorial theorem of Helley. (Of course we also employ P.M. Grubers idea of applying the theorem of affine geometry in the final step of our proofs of various lemmas, as has become standard in the literature).\\ \\
We denote by $\mathbb{E}^{d}$ the $d$-dimensional Euclidean space endowed with the usual inner product $\langle .| .\rangle$ and by $
\mathscr{C}(\mathbb{E}^{d})$ the {\it space of convex bodies in ${\mathbb E}^{d}$ i.e. the space of all convex compact sets $C\subset \mathbb{E}^{d}$ including the empty set}.
Let $\conv(X)$ denote the convex hull, $\aff(X)$ the affine hull, $\interior(X)$ the interior and $\relint(X)$ the relative interior (in $\aff(X)$) of an arbitrary set $X\subseteq {\mathbb E}^{d}$. We further denote by $\dim(C)$ or $\dim C$ the dimension of a convex set $C\subseteq {\mathbb E}^{d}$ i.e. the dimension of the affine subspace spanned by $C$ and regard the empty set as $(-1)$-dimensional.\\ \\
Note that $(\mathscr{C}(\mathbb{E}^{d}), \wedge, \vee )$ forms a lattice with respect to the operations
\begin{equation*}
C \wedge D = C \cap D \hbox{ and } C \vee D = \conv(C \cup D) \hbox { for } C,D \in \mathscr{C}.
\end{equation*}
We will use the notation $C \vee D := {\conv} (C \cup 
D)$ also for arbitrary (i.e. not necessarily convex) subsets $C,D$ of ${\Bbb{E}}^d$. We recall that
a \emph{(lattice) homomorphism} from some lattice $( K, \wedge, \vee)$ to some lattice $( L, \sqcap, \sqcup )$ is a mapping $\Phi: K \to L$ such that
\begin{equation*}
\Phi(C \wedge D) = \Phi(C) \sqcap \Phi(D) \hbox{ and } \Phi(C \vee D) = \Phi(C) \sqcup \Phi(D),
\end{equation*}
while an anti-homomorphism $\Lambda:K\to L$ fulfils
\begin{equation*}
{\Lambda}(C \wedge D) = {\Lambda}(C) \sqcup {\Lambda}(D) \hbox{ and } {\Lambda}(C \vee D) = {\Lambda}(C) \sqcap {\Lambda}(D).
\end{equation*}
The simplest lattice homomorphisms are functions $\Phi:{\mathscr C}({\mathbb E}^{c})\to {\mathscr C}({\mathbb E}^{d})$ that map all convex compact sets---including the empty set---onto one and the same set. We call such homomorphisms {\it trivial}.\\ \\
Our main result Theorem \ref{Theorem Homomorphisms} provides a complete characterisation of the non-trivial homomorphisms $\Phi:\mathscr{C}(\mathbb{E}^{c})\to \mathscr{C}(\mathbb{E}^{c+1})$ provided that $c\geq 3$. It shows that a non-trivial homomorphism $\Phi:\mathscr{C}(\mathbb{E}^{c})\to \mathscr{C}(\mathbb{E}^{c+1})$ either raises the dimension of all non-empty convex bodies of $E^{c}$ by $1$ or keeps the dimension of all convex bodies constant.
\\ \\
If $\Phi$ keeps the dimension constant, we obtain that $\phi$ is induced by an affine bijection $\phi:{\mathbb E}^{c}\to H\subseteq {\mathbb E}^{c+1}$ ---which parallels the main result on endomorphisms provided in \cite{Gr91}.
In case that $\Phi$ is dimension raising we further have to distinguish the cases that $\Phi$ maps the empty set to the empty set, or that $\Phi(\emptyset)$ equals some one-point set $\{ o \} \subset {\mathbb E}^{c+1}$. The case that $\Phi(\emptyset)=\emptyset$ finally splits into the cases that $\Phi$ maps all one-point sets to proper parallel line-segments or that the images of any two distinct one-point sets are not parallel. In any case we eventually obtain that the representation of $\Phi$ involves an affine bijection $\phi:{\mathbb E}^{c}\to H\subseteq {\mathbb E}^{c+1}$ proving that even homomorphisms that raise the dimension by $1$ are rather rigid.\\ \\
The proof of Theorem \ref{Theorem Homomorphisms} is split into three lemmas: Lemma \ref{Gruberfalllemma} treats the case that $\Phi$ is not dimension raising, Lemma \ref{E-empty-to-point-lemma} the case that $\Phi(\emptyset)\neq \emptyset$ and Lemma \ref{empty-to-empty-lemma} the parallel as well as the non-parallel case of a dimension raising $\Phi$ for which $\Phi(\emptyset)=\emptyset$.\\ \\
The most important ingredient in the proof of Lemma \ref{E-empty-to-point-lemma} and Lemma \ref{empty-to-empty-lemma}---and thus of Theorem \ref{Theorem Homomorphisms}---are considerations concerning the relationship between the dimension of a convex body $C\subset {\mathbb E}^{c}$ and its image $\Phi(C)\subset {\mathbb E}^{d}$ under some arbitrary lattice homomorphism $\Phi:{\mathscr C}({\mathbb E}^{c})\to {\mathscr C}({\mathbb E}^{d})$. The lower bound for the dimension of $\Phi(C)$ is obtained in Proposition \ref{dimprop} by a simple application of Radon's theorem, while the upper bound is obtained in Lemma \ref{Dimlemma} by a less simple application of second countability of ${\mathbb E}^{d}$.\\ \\
While the proofs of Lemma \ref{Gruberfalllemma} and Lemma \ref{E-empty-to-point-lemma} are---up to dimension considerations addressed above---rather elementary, the proof of Lemma \ref{empty-to-empty-lemma} is not. First it involves with Lemma \ref{radon-affine-dependence-lemma} (and thus Proposition \ref{radon-affine-dependence-proposition}) an argument---based on a further application of Radon's theorem---ensuring that even in the dimension raising case affine dependence is preserved in a certain sense by lattice homomorphism. Secondly we have to use two transversality theorems that are displayed in Appendix \ref{transversality-theorems} and rely on Helly's theorem. Note, that to our knowledge Theorem \ref{transversality theorem}---a transversality result for line segments directed towards a common point $o$---has not occurred in the literature before.\\ \\ 
As an application of Theorem \ref{Theorem Homomorphisms} we prove that the homomorphisms from ${\mathscr C}({\mathbb E}^{c})$ to certain spaces of convex functions $f:\mathbb{E}^{c}\to [0,+\infty]$ are induced by affine transformations (Theorem \ref{Homomorphismen Satz Anwendungen}) while the anti-homomorphisms from ${\mathscr C}({\mathbb E}^{c})$ to certain spaces of convex functions $f:\mathbb{E}^{c}\to (-\infty,+\infty]$ are induced by the concatenation of affine transformations with the Legendre-Fenchel transformation (Theorem \ref{anti-homo-theorem}). Note that the anti-homomorphisms and homomorphisms under consideration are again dimension raising.\\ \\
In Appendix \ref{endomorphisms} we restate the main result of \cite{Gr91} for dimensions $ \geq 2$ and show that it is a simple consequence of the dimension considerations in Section \ref{dimension-arguments} and Proposition \ref{Prop Bedingungen für Phi}.\\ \\
Given a mapping $\Phi:{\mathscr C}({\mathbb E}^{c})\to {\mathscr C}({\mathbb E}^{d})$ we use (cf. \cite{Gr91}) the term $\Phi(x)$ as a synonym for $\Phi(\{ x \})$. For $x,y \in {\mathbb E}^{c}$ and $C\in {\mathscr C}({\mathbb E}^{c})$ we use  $x\wedge y$, $x\vee y$, $x\wedge C$ and $x\vee C$ as synonyms for $\{ x \} \wedge \{ y \}$, $\{ x \} \vee \{ y \}$, $\{ x\} \wedge C$ and $\{ x \}\vee C$. We let $\bigvee_{r\in R} C_{r} := \conv(\bigcup_{r\in R} C_{r})$ and $\bigwedge_{r\in R} C_{r} :=\bigcap_{r\in R} C_{r}$ and further write $(\forall x,y,z\in X)\dots $ instead of $(\forall x)(\forall y)(\forall z) ((x\in X\ \hbox{and}\ y\in X\ \hbox{and}\ z\in X)\Rightarrow \dots )$. 

\begin{remark}
Note that any lattice homomorphism is order preserving while any anti-homomorphism is order reversing. We will use this fact throughout the article without further reference.
\end{remark}
\begin{remark}
The dependence structure of the obtained results and the organization of the article can be briefly summarised as follows: We use the results of the Sections \ref{preliminaries} and \ref{dimension-arguments} together with Proposition \ref{Prop Bedingungen für Phi} to prove the Lemmas of Section \ref{section-no-dim-gape} and \ref{section-dim-gape}. (Note that for the proof of Lemma \ref{empty-to-empty-lemma} in Section \ref{section-dim-gape} we make additional use of the transversality theorems of Appendix \ref{transversality-theorems}.) The Lemmas of the sections \ref{section-no-dim-gape} and \ref{section-dim-gape} together imply the main characterization Theorem \ref{Theorem Homomorphisms} for homomorphisms from ${\mathscr C}({\mathbb E}^{c})$ to ${\mathscr C}({\mathbb E}^{c+1})$. The reason for including the main result, Theorem \ref{Theorem Homomorphisms}, into Section \ref{section-main-theo} is twofold: First of all we wanted to place the statement of the main result prior to the technical details of its proof, secondly it seems appropriate to state Theorem \ref{Theorem Homomorphisms} right after Proposition \ref{Prop Bedingungen für Phi} since the conjunction of Theorem \ref{Theorem Homomorphisms} and Proposition \ref{Prop Bedingungen für Phi} provides a further characterisation of the homomorphisms from ${\mathscr C}({\mathbb E}^{c})$ to ${\mathscr C}({\mathbb E}^{c+1})$ displayed as Corollary \ref{char-of-homo-corollary}.  The results obtained in Section \ref{section-applications} are consequences of Theorem \ref{Theorem Homomorphisms}.
\end{remark}

\section{Preliminaries}\label{preliminaries}

We make use of the following preliminary results.

\begin{notation}
For $a,b\in {\mathbb E}^{c}$ we let $[a,b]:=a\vee b$, $(a,b] := [a,b]\setminus \{ a \}$, $[a,b):=(b,a]$ and $(a,b):=(a,b]\cap [a,b)$ and call these sets line-segments provided that they are not empty. We call them proper line segments or non-degenerate line-segments provided they are one-dimensional.  
\end{notation}

\begin{proposition}\label{Geradenprop}
Let $\mathcal{G}$ be a set of pairwise intersecting lines in $\mathbb{E}^{d}$ that contains three distinct lines $g_{0},g_{1},g_{2}$ as elements such that $g_{0}\nsubseteq \aff(g_{1},g_{2})$.
Then
\begin{equation*}
\bigcap \mathcal{G} = \{ o \}
\end{equation*}
for some $o\in {\mathbb E}^{d}$.\hfill $\Box$
\end{proposition}

\begin{notation}
Given points $s,p,q\in {\mathbb E}^{d}$ we let $\triangle(s,p,q):=s\vee p \vee q$ and call $\triangle(s,p,q)$ a triangle. The triangle $\triangle(s,p,q)$ is called non-degenerate if it is a two dimensional convex body i.e. if $\triangle(s,p,q)$ is neither a point nor a proper line-segment.
\end{notation}

\begin{proposition}\label{Dreiecksprop}
Let $\triangle (s,p,q)$ be a non-degenerate triangle, let $r \in \triangle (s,p,q)$, $r \notin [p,q]$ and $o \in (p,q)$. Then 
\begin{equation*}
\relint([s,o] \cap \triangle (r,p,q))\setminus \{ o \} \neq \emptyset.
\end{equation*}\hfill $\Box$
\end{proposition}

\begin{notation}
We call any 2-dimensional affine space a plane.
\end{notation}

\begin{proposition}\label{existence-of-second-line}
Let $x\in (s_{1},s_{2})$ and let $s_{1}\in (u_{1},v_{1})$, with $(u_{1},v_{1})\nparallel (s_{1},s_{2})$. Then there exist $u_{2},v_{2}$ such that $s_{2}\in (u_{2},v_{2})$ and
\begin{equation*}
\begin{split}
\{ x \} = (u_{1},v_{2})\cap (v_{1}, u_{2})=(u_{1},v_{2})\cap (s_{1},s_{2})=(v_{1},u_{2})\cap (s_{1},s_{2}),\\
\{ x \} = [u_{1},v_{2}]\cap [v_{1}, u_{2}]=[u_{1},v_{2}]\cap [s_{1},s_{2}]=[v_{1},u_{2}]\cap [s_{1},s_{2}].
\end{split}
\end{equation*}
Further $x, s_{1},s_{2},u_{1},v_{1},u_{2},v_{2}$ are located in one plane.\hfill $\Box$
\end{proposition}

\begin{remark}
Note that the triangles $\triangle (o,p,q)$ and $\triangle (o,r,s)$ in the following proposition are not contained in a common plane.
\end{remark}

\begin{proposition}\label{relint-two-triangles-intersection-prop}
Let $\triangle(o,p,q)$ and $\triangle(o,r,s)$ be two non-degenerate triangles that intersect in the non-degenerate line-segment $[o,y]$ i.e. suppose that\linebreak $[o,y]=\triangle (o,p,q) \cap \triangle (o,r,s)$. Suppose further that $y\in \relint(\triangle(o,p,q))$ and $y\notin [o,r]\cup [o,s]$ . Then $y\in (r,s)$.\hfill $\Box$
\end{proposition}

\begin{proposition}\label{different-rays-two-dim-prop}
Let $r$ be a ray emanating from some point $o$ and let $g$ be a line through $o$ such that $r\nsubseteq g$. Let $s$ be a non-degenerate line segment on $g$ and let $C_1, C_2$ be convex sets with $(C_{i}\cap r) \setminus o\neq \emptyset$. Then $\dim (C_1 \vee s) \wedge (C_2 \vee s) \geq 2$.\hfill $\Box$
\end{proposition}

\begin{proposition}\label{parallel-lines-two-dim-prop}
Let $g, \tilde{g}$ be distinct parallel lines, $s\subseteq g$ a non-degenerate line-segment and let $\emptyset \neq C_1, C_2 \subseteq \tilde{g}$ be convex. Then $\dim (C_1 \vee s) \wedge (C_2 \vee s) = 2$.\hfill $\Box$
\end{proposition}

\begin{proposition}\label{quadrangle-prop}
Let $F$ be a plane and let $a_{1},a_{2},b_{1},b_{2}\in F$ be such that $(a_{1},b_{2})\cap (a_{2},b_{1}) = \{ z \}$ for some $z\in F$. Then for any $o\in F\setminus \{ z \}$ 
\begin{equation*}
(\exists e_{1}, e_{2}, w \in \{ a_{1}, a_{2}, b_{1}, b_{2} \})\ \hbox{ s.t. }\  e_{1}\neq w\neq e_{2} \hbox{ and } \triangle(e_{1},e_{2},o)\cap [w,o]\supsetneq \{ o \}. 
\end{equation*}
\end{proposition}

\begin{proof}
Hint: The Proposition is proved by distinguishing the cases:
\begin{gather*}
1.\ \ \ \{e_{1},e_{2}\} = \{ a_{1}, b_{2}\}\hbox{ and }w=a_{2}\ \ \ \ \ \ 2.\ \ \ \{e_{1},e_{2}\} = \{ a_{1}, b_{2}\}\hbox{ and }w=b_{1}\\
3.\ \ \ \{e_{1},e_{2}\} = \{ a_{2}, b_{1}\}\hbox{ and }w=a_{1}\ \ \ \ \ \ 4.\ \ \
\{e_{1},e_{2}\} = \{ a_{2}, b_{1}\}\hbox{ and }w=b_{2}
\end{gather*}
For each of these cases one defines a region $O_{i}$ (with $i$ indexing the case) by\linebreak $O_{i} := \{ o \mid \triangle(e_{1},e_{2},o)\cap [w,o]\supsetneq \{ o \} \}$. One then shows that $F\setminus \{ z \}= \bigcup_{i=1}^{4} O_{i}$.
\end{proof}

\begin{proposition}\label{ray hyperlane intersection existance gamma}
Given rays $r_i$ emanating from some point $o$ and parallel hyperplanes $G \neq H$ such that all rays intersect $G$ and $H$. Suppose further that $G$ and $o$ are located on the same side of $H$ and $o \notin G$ i.e. $G$ lies strictly between $H$ and $o$. Then $\exists \gamma \in (0,1)$ independent of the parameter $i$ such that for $\{ p_{i} \} := r_{i}\cap H$ we have $\gamma p_{i} + (1- \gamma)o \in G$.\hfill $\Box$
\end{proposition}

\begin{notation}
We use $C\parallel D$ to express that $C$ and $D$ are parallel convex sets i.e. that $\aff(C)\subseteq \aff(D) + v$ or $\aff(D)\subseteq \aff(C) + v$ for some appropriate vector $v$. Note that $\parallel$ is not a transitive relation in general, but that $\parallel$ is transitive on the subspaces of convex sets of equal dimension.
\end{notation}

\begin{proposition}\label{Prop phi+ = phi- + v}
Let $G, H$ be parallel, distinct affine subspaces of $\mathbb{E}^{d}$ and let $\phi^{+}: \mathbb{E}^{c} \to G, \phi^{-}: \mathbb{E}^{c} \to H$ be such that 
\begin{equation*}
(\forall x,y\in {\mathbb E}^{c})\quad [\phi^{-}(x), \phi^{+}(x)]\ \parallel\ [\phi^{-}(y), \phi^{+}(y)].
\end{equation*}
Then there exists some $v \in \mathbb{E}^{d}$ such that $(\forall x \in \mathbb{E}^{c})\ \phi^{+}(x) = \phi^{-}(x) +v$. \hfill $\Box$
\end{proposition}

\begin{proposition}\label{affine-bijection-prop}
Let $F$ and $H$ be parallel distinct hyperplanes in ${\mathbb E}^{c+1}$ and let $o\in {\mathbb E}^{c+1}\setminus (F\cup H)$. Let $\psi:{\mathbb E}^{c}\to F$ be an affine bijection and let $\phi:{\mathbb E}^{c}\to H$ be such that $(\forall v\in {\mathbb E}^{c})\ \aff(\psi(v),\phi(v))\ni o$. Then $\phi$ is an affine bijection.
\end{proposition}

\begin{proposition}\label{affine-bijection-prop-parallel}
Let $F$ and $H$ be parallel distinct hyperplanes in ${\mathbb E}^{c+1}$ and let $g$ be a line that intersects $F$ as well as $H$. Let $\psi:{\mathbb E}^{c}\to F$ be an affine bijection and let $\phi:{\mathbb E}^{c}\to H$ be such that $(\forall v\in {\mathbb E}^{c})\ \aff(\psi(v),\phi(v))\parallel g$. Then $\phi$ is an affine bijection.
\end{proposition}

\begin{definition}
Let $X\subseteq {\mathbb E}^{c}$. We say that $\phi:X\to {\mathbb E}^{d}$ preserves the order of points if
\begin{equation*}
(\forall x,y,z\in X)\quad y\in (x,z) \implies \phi(y)\in (\phi(x),\phi(z)) 
\end{equation*}
\end{definition}

\begin{definition}
We say that a family of points is collinear if the points of the family are located on some common line. Let $X\subseteq {\mathbb E}^{c}$. We say that $\phi:X\to {\mathbb E}^{d}$ is a collinear mapping if the image of any collinear family of points under $\phi$ is again collinear.
\end{definition}

\begin{proposition}\label{order-preserving=>collinear}
Let $X\subseteq {\mathbb E}^{c}$ and $\phi:X\to {\mathbb E}^{d}$. If $\phi$ preserves the order of points then $\phi$ is a collinear mapping.\hfill $\Box$
\end{proposition}

\begin{lemma}[Main theorem of affine geometry]\label{main-theorem-aff-geometry}
Let $c\geq 2$ and let $\phi:{\mathbb E}^{c}\to {\mathbb E}^{d}$ be a collinear injective mapping. Then $\phi$ is an affine bijection onto its image.
\end{lemma}

\begin{proof}
For a proof consult \cite{Le58}.
\end{proof}

\begin{corollary}\label{corollary-aff-geometry}
Let $c\geq 2$ and let $\phi:{\mathbb E}^{c}\to {\mathbb E}^{d}$ be an injective mapping that preserves the order of points. Then $\phi$ is an affine bijection onto its image.
\end{corollary}

\begin{proof}
The corollary is an immediate consequence of Proposition \ref{order-preserving=>collinear} and Lemma \ref{main-theorem-aff-geometry}.
\end{proof}

\begin{proposition}\label{preservation-of-point-order-prop}
Let $F\subseteq {\mathbb E}^{d}$ be a hyperplane, let $V\subseteq {\mathbb E}^{c}$ be an arbitrary set, let $(r_{v})_{v\in V}$ be a family of rays emanating from some common point $o\in {\mathbb E}^{d}\setminus F$ such that $v\neq w\ \Rightarrow r_{v}\neq r_{w}$. Let $(\Phi(v))_{v\in V}$ be a family of convex bodies in ${\mathbb E}^{d}$ such that $\Phi(v)\subset r_{v}\setminus \{  o \}$ and $\Phi(v)\cap F = \{ p_{v} \}$ for some $p_{v}\in {\mathbb E}^{d}$. Suppose further that $v\in (x,y) \Rightarrow \Phi(v) \subseteq \Phi(x) \vee \Phi(y)$. Then the function $\psi:V\to {\mathbb E}^{d}$ given by $\psi(v)=p_{v}$ preserves the order of points.   
\end{proposition}

\begin{proof}
Let $v\in (x,y)$. Then $\psi(v)\in \Phi(v)\cap F\subseteq (\Phi(x)\vee \Phi(y))\cap F = [ \psi(x),\psi(y)]$. Since $\Phi(v)\cap \Phi(x) = \Phi(v)\cap \Phi(y) = \emptyset$ and thus $\psi(x)\neq \psi(v)\neq \psi(y)$ we obtain that $\psi(v)\in (\psi(x),\psi(y))$ i.e. $\psi$ preserves the order of points.
\end{proof}

\begin{proposition}\label{preservation-of-point-order-prop-prallel}
Let $F\subseteq {\mathbb E}^{d}$ be a hyperplane, let $V\subseteq {\mathbb E}^{c}$ be an arbitrary set, let $(g_{v})_{v\in V}$ be a family of parallel lines such that $v\neq w\ \Rightarrow g_{v}\neq g_{w}$. Let $(\Phi(v))_{v\in V}$ be a family of convex bodies in ${\mathbb E}^{d}$ such that $\Phi(v)\subset g_{v}$ and $\Phi(v)\cap F = \{ p_{v} \}$ for some $p_{v}\in {\mathbb E}^{d}$. Suppose further that $v\in (x,y) \Rightarrow \Phi(v) \subseteq \Phi(x) \vee \Phi(y)$. Then the function $\psi:V\to {\mathbb E}^{d}$ given by $\psi(v)=p_{v}$ preserves the order of points.
\end{proposition}

\begin{proof}
The proof is the same as that of Proposition \ref{preservation-of-point-order-prop}.
\end{proof}

\section{Elementary facts and Dimension arguments}\label{dimension-arguments}

\begin{proposition}\label{range-of-Phi-determined-by-image-of-one-point-sets}
Let $\Phi:{\mathscr C}({\mathbb E}^{c})\to {\mathscr C}({\mathbb E}^{d})$ be an arbitrary lattice homomorphism. Then 
\begin{equation*}
(\forall C\in {\mathscr C}({\mathbb E}^{c}))\quad \phi(C)\subseteq \conv(\{ \Phi(x) \mid x\in {\mathbb E}^{c}\})
\end{equation*}
and thus especially $\{ \Phi(x) \mid x\in {\mathbb E}^{c}\} \subseteq F$ with $F\subseteq {\mathbb E}^{d}$ some convex set (e.g. $F$ an affine space etc.) implies that $\Phi(C)\subseteq F$ for any $C \in {\mathscr C}({\mathbb E}^{c})$.   
\end{proposition}

\begin{proof}
Given $C\in {\mathscr C}({\mathbb E}^{c})$ let $P=x_{1}\vee\dots\vee x_{n}\subset {\mathbb E}^{c}$ be a polytope such that $C\subseteq P$. Then 
\begin{equation*}
\Phi(C)\subseteq \Phi(P)=\Phi(x_{1})\vee\dots\vee \Phi(x_{n})\subseteq \conv(\{ \Phi(x) \mid x\in {\mathbb E}^{c}\})
\end{equation*}
which proves the proposition.
\end{proof}

\begin{proposition}\label{injectivity-on-points-prop} Let $c\geq 2$ and let $\Phi:{\mathscr C}({\mathbb E}^{c})\to {\mathscr C}({\mathbb E}^{d})$ be a non-trivial lattice homomorphism. Then $x\mapsto \Phi(x)$ is an injection on the family of one-point sets and $(\forall x\in {\mathbb E}^{c})\ \Phi(\emptyset) \subsetneq \Phi(x)$ and thus in particular, $(\forall x \in 
{\mathbb{E}}^c) \,\, \Phi (x) \neq \emptyset $.
\end{proposition}

\begin{proof}
Indirect: Suppose that $ u,v \in \mathbb{E}^{c}$ were distinct points with $u \neq v$ and $\Phi(u) = \Phi(v)$, then we have 
\begin{equation*}
\Phi(u) = \Phi(v) = \Phi(u) \wedge \Phi(v) = \Phi(u \wedge v) = \Phi(\emptyset).
\end{equation*}
By \cite[Sections 2.2 and 2.3]{Gr91} we obtain from $(\exists u\in {\mathbb E}^{c})\ \Phi(u) = \Phi(\emptyset)$ that $\Phi$ is a trivial homomorphism.\footnote{For the sake of completeness we summarise the arguments of \cite[Sections 2.2 and 2.3]{Gr91}. Suppose that $A:=\{ x\in {\mathbb E}^{c} \mid \phi(x)=\phi(\emptyset)\}\neq \emptyset$. It is easily seen that $A$ is convex. Further we assume that $A\subsetneq {\mathbb E}^{c}$ since otherwise we would already know that $\Phi$ is trivial. If $A=\{ p \}$ for some $p\in {\mathbb E}^{c}$ then for $z\neq p$ and $y\in (p,z)$ we have $\Phi(y), \Phi(z)\supsetneq \Phi(\emptyset)=\Phi(p)$, thus $\Phi(y)\subseteq \Phi(p\vee z) =\Phi(z)$ and thus finally $\Phi(\emptyset)=\Phi(y\wedge z)=\Phi(y)\wedge \Phi(z) = \Phi(y)\neq\Phi(\emptyset)$ which is contradictory. Thus there exist $x,y\in A$ with $x\neq y$. Since $A\subsetneq {\mathbb E}^{c}$ is convex and $c\geq 2$ there exists further some $w\in {\mathbb E}^{c}\setminus A$ such that $\{ w \} = (x,u)\cap (y,v)$ for appropriate $u,v\in {\mathbb E}^{c}\setminus A$. Thus $\Phi(w)=(\Phi(x)\vee \Phi(u))\wedge (\Phi(y)\vee \Phi(v)) = (\Phi(u)\wedge \Phi(v))\vee \Phi(\emptyset) = \Phi(\emptyset)\vee \Phi(\emptyset) = \Phi(\emptyset)$ contradicting $w\notin A$.} \emph{Contradiction.}
\end{proof}

\begin{proposition}\label{x-notin-C-prop}
Let $c\geq 2$ and let $\Phi:{\mathscr C}({\mathbb E}^{c})\to {\mathscr C}({\mathbb E}^{d})$ be a non-trivial lattice homomorphism. Let $C\in {\mathscr C}({\mathbb E}^{c})$ and $x\in {\mathbb E}^{c}$ be such that $x\notin C$. Then $\Phi(x)\nsubseteq \Phi(C)$ and thus $\Phi(x\vee C)\setminus \Phi(C)\neq \emptyset$.
\end{proposition}

\begin{proof}
Since $\Phi$ is a non-trivial homomorphism we obtain by application of Proposition \ref{injectivity-on-points-prop} that 
\begin{equation*}
\Phi(x)\wedge \Phi(C) = \Phi(x\wedge C)= \Phi(\emptyset) \subsetneq \Phi(x)
\end{equation*}
from which the result follows.
\end{proof}

\begin{proposition} \label{family-at-most-count.prop}
Let $\mathcal{G}$ be a family of pairwise disjoint subsets of $\mathbb{E}^{d}$ such that $(\forall G \in \mathcal{G})\ \interior(G) \neq \emptyset$.
Then $\mathcal{G}$ is (at most) countable.
\end{proposition}

\begin{proof}
Since $\mathbb{E}^{d}$ is second countable i.e. possesses a countable base for its topology, any disjoint family of open sets has to be countable.
\end{proof}

\begin{proposition}\label{Y-at-most-countable-prop}cf. \cite[Section 2.4]{Gr91}
Let $D \subset \mathbb{E}^{d}$ be a compact set and let $\mathcal{Y}$ be a family of non empty sets $Y$ such that $Y\subseteq \aff(D)$ and the family\linebreak $\{ (Y\vee D)\setminus D \mid Y\in \mathcal{Y}\}$ consists of pairwise disjoint non-empty sets. Then $\mathcal{Y}$ is at most countable.
\end{proposition}

\begin{proof}
Let $F$ denote the affine hull of $D$ and let $int_{F}(Z)$ denote the relative interior of any set $Z \subseteq F$ with respect to $F$. Since $D$ is compact and thus closed, $Y\vee D$ is convex and $(Y\vee D)\setminus D\neq \emptyset$, therefore we have that
\begin{equation*}
int_{F}((Y\vee D) \setminus D) \neq \emptyset \label{eqn:relint}
\end{equation*}
Since $F$ is second countable the result follows.
\end{proof}

\begin{remark}
Note that Proposition \ref {Y-at-most-countable-prop} includes the case that $D=\emptyset$, in which it is trivially fulfilled.
\end{remark}

\begin{proposition}\label{dimprop}
Let $\Phi:{\mathscr C}({\mathbb E}^{c})\to {\mathscr C}({\mathbb E}^{d})$ be a non-trivial lattice homomorphism such that $\Phi(\emptyset)=\emptyset$ and let $c\geq 2$. Then for $C\in {\mathscr C}({\mathbb E}^{c})$ we have 
\begin{equation*}
\dim C\leq \dim \Phi(C)
\end{equation*}
and in particular $c \leq d$.
\end{proposition}

\begin{proof}
Let $\dim C=m+1$, let $S\subseteq C$ be a set consisting of $m+2$ points in general position. We proceed indirect: Suppose that $\dim \Phi(C) \leq m$ and thus $\dim(\Phi(\conv( S)))\leq m$. By Radon's Theorem \cite[Section 3.2]{Gr07} and Proposition \ref{injectivity-on-points-prop} (recall that $\Phi$ is a non-trivial lattice homomorphism) there exist disjoint sets $R$ and $B$ such that $R\cup B=S$ and 
\begin{equation}\label{By-radnon-neq-emptyset}
\Phi(\conv(R))\wedge \Phi(\conv(B))= \left (\bigvee_{x\in R} \Phi(x)\right ) \wedge \left (\bigvee_{x\in B} \Phi(x)\right )\neq \emptyset.
\end{equation}
Since $S$ consists of points in general position we also obtain 
\begin{equation*}
\Phi(\conv(R))\wedge \Phi(\conv(B))=\Phi(\conv(R)\wedge \conv(B))=\Phi(\emptyset)=\emptyset.
\end{equation*} 
Contradicting (\ref{By-radnon-neq-emptyset}).   
\end{proof}

\begin{lemma} [Dimension Lemma] \label{Dimlemma}
Let $\Phi:{\mathscr C}({\mathbb E}^{c})\to {\mathscr C}({\mathbb E}^{d})$ be a non-trivial lattice homomorphism and $c\geq 2$. Then for $C\in {\mathscr C}({\mathbb E}^{c})$ we have
\begin{equation*}
\dim (C)\leq c-2\ \Longrightarrow\ \dim (\Phi(C)) \leq \dim (C) + d - c.
\end{equation*}
\end{lemma}

\begin{remark} The proof of the Dimension Lemma is based on Propositions \ref{family-at-most-count.prop} and \ref{Y-at-most-countable-prop} i.e. on the topological fact that any real affine space possesses a countable base for its usual (euclidean) topology and thus that any family of disjoint open subsets of a real affine space is countable. Proposition \ref{Y-at-most-countable-prop} has already implicitly been used in \cite[Section 2.4]{Gr91}. However, the full power of this simple topological fact---demonstrated by the derivation of Lemma \ref{Dimlemma} below---has not been exploited before.
\end{remark}

\begin{remark} \label{Remark c geq 3}
Note that even in the case that $\dim C =1$ Lemma \ref{Dimlemma} only applies for spaces $\mathscr{C}(\mathbb{E}^{c})$ with $c \geq 3$. Thus, with exception of Lemma \ref{Gruberfalllemma}, Theorem \ref{Grubersatz} and some propositions, we consider homomorphisms $\Phi: \mathscr{C}(\mathbb{E}^{c}) \to \mathscr{C}(\mathbb{E}^{c+1})$ for $c \geq 3$ exclusively.
\end{remark}

\begin{proof}[Proof of Lemma \ref{Dimlemma}]
Indirect: Let $ C \in \mathscr{C}(\mathbb{E}^{c})$ be such that $ \dim C \leq c-2 $ and assume that 
\begin{equation} \label{Dimassumption}
\dim (\Phi(C)) > \dim (C) +d-c.\\
\end{equation}
Let $ L := \aff(C)$ and note that $\aff(\emptyset)=\emptyset$. Denote by $L^\perp $ the orthogonal complement of $L$ in ${\mathbb E}^{c}$ i.e. denote by $L^{\perp}\subseteq {\mathbb E}^{c}$ the maximal vector-space such that $L$ and $L^{\perp}$ are orthogonal i.e. let 
\begin{equation*}
L^{\perp}:=\{ y \in {\mathbb E}^{c} \mid x,x_{0}\in L\ \hbox{ implies that } \langle x-x_{0} | y\rangle = 0 \}.
\end{equation*}
Note that $\emptyset^{\perp} = \{ p \}^{\perp} = {\mathbb E}^{c}$ for any $p\in {\mathbb E}^{c}$. Denote further by $\mathbb{S}$ the unit sphere in $L^{\perp}$ with center $L\cap L^{\perp}$ in case that $L\neq\emptyset$ and denote by $\mathbb{S}$ the unit sphere of ${\mathbb E}^{c}$ with center $0$ in case that\footnote{Letting ${\mathbb S}$ denote the unit sphere in ${\mathbb E}^{c}$ is the only additional ingredient necessary to make our argument work in the seemingly exceptional case that $C$ is $(-1)$-dimensional i.e. in the case that $L:=\aff(C)=C=\emptyset$.}
$L=\emptyset$.\\ \\
Since $\dim C \leq c-2 $ one has $\dim L \leq c-2 $, hence $\dim L^\perp \geq 2$ and $ \dim \mathbb{S} \geq 1 $ and therefore we have:
\begin{align}\label{alg:Suncout}
\mathbb{S} \hbox{ is uncountable}. 
\end{align}
Since $x\in {\mathbb S}$ implies that $x\notin C$ we obtain by application of Proposition \ref{x-notin-C-prop} that
\begin{equation}\label{x-notin-S-implication}
\forall x\in {\mathbb S}\quad \Phi(x\vee C)\setminus \Phi(C)\neq \emptyset.
\end{equation}
Further 
\begin{equation*}
(\forall x,y \in \mathbb{S})\quad x \neq y\Rightarrow ( C \vee x ) \wedge ( C \vee y ) = C.
\end{equation*}
Since $\Phi$ is a lattice homomorphism the last formula implies 
\begin{equation*}
(\forall x,y \in \mathbb{S})\quad x \neq y\Rightarrow (\Phi(x)\vee \Phi(C))\wedge (\Phi(y)\vee \Phi(C)) = \Phi(C)
\end{equation*}
which we reformulate as
\begin{equation}\label{E-almost-disjoint}
(\forall x,y \in \mathbb{S})\quad x \neq y\Rightarrow ((\Phi(x)\vee \Phi(C))\setminus \Phi(C)) \wedge ((\Phi(y)\vee \Phi(C))\setminus \Phi(C)) = \emptyset.
\end{equation}
From (\ref{x-notin-S-implication}) and (\ref{E-almost-disjoint}) we obtain that $( (\Phi(x) \vee \Phi(C)) \setminus \Phi(C))_{x \in {\mathbb S}}$ is an indexed family of pairwise disjoint non-empty sets and thus by application of Proposition \ref{Y-at-most-countable-prop} with $D=\Phi(C)$ and ${\cal Y} = \{ \phi(x)\mid \Phi(x) \subseteq \aff(\Phi(C)) \}$ that 
\begin{equation}\label{F-countable}
F:= \lbrace x\in {\mathbb S} \mid \Phi(x) \subseteq \aff(\Phi(C)) \rbrace \hbox{ is (at most) countable.}
\end{equation} 
Further (\ref{F-countable}) implies that 
\begin{equation}\label{exists-x-in-S-dim->}
(\forall x\in {\mathbb S}\setminus F)\quad ( \dim \Phi(C\vee x) > \dim \Phi(C) ).
\end{equation}
Since ${\mathbb S}$ is by (\ref{alg:Suncout}) uncountable and $F$ is countable we obtain that
\begin{equation}\label{SF-uncountable}
{\mathbb S}\setminus F \hbox{ is uncountable (and in particular not empty)}. 
\end{equation}
By induction we obtain from (\ref{exists-x-in-S-dim->}) the existence of a set $C\in {\mathscr C}({\mathbb E}^{c})$ such that $\dim(C)=c-2$ and the assertions (\ref{E-almost-disjoint}),  (\ref{F-countable}), (\ref{SF-uncountable}), (\ref{exists-x-in-S-dim->}) and (\ref{Dimassumption}) still hold.\\ \\ 
Since $\dim(C) = c-2$ the hypothesis (\ref{Dimassumption}) specialises to $\dim \Phi(C) > d-2$ and thus it remains to distinguish the cases $\dim \Phi(C)=d$ and $\dim \Phi(C)=d-1$.\\ \\
The case $\dim \Phi(C)=d$ contradicts the conjunction of (\ref{exists-x-in-S-dim->}) and (\ref{SF-uncountable}) since the existence of some $x$ such that $\dim(\Phi(C\vee x))>d$ is impossible.\\ \\
In the case that $\dim (\Phi(C)) =d-1$ we obtain by (\ref{exists-x-in-S-dim->}) that
\begin{equation*}
(\forall x\in {\mathbb S}\setminus F)\quad (\dim \Phi(C \vee x) = d) 
\end{equation*}
and thus
\begin{equation*}
(\forall x\in {\mathbb S}\setminus F)\quad \interior( \Phi(C \vee x)\setminus \Phi(C))\neq \emptyset.
\end{equation*}
The last equation contradicts---together with (\ref{SF-uncountable}) and (\ref{E-almost-disjoint})---Proposition \ref{family-at-most-count.prop}.
\end{proof}

\begin{remark}
Although the dimension of convex sets under a homomorphism may rise in accordance with Lemma \ref{Dimlemma}, homomorphisms always preserve affine dependence in the following sense:
\end{remark}

\begin{proposition}\label{radon-affine-dependence-proposition}
Let $d\geq c\geq 2$ and let $\Phi:\mathscr{C}({\mathbb E}^{c})\to \mathscr{C}({\mathbb E}^{d})$ be a non-trivial lattice homomorphism. Let a $(c+2)$-point set $S\subseteq {\mathbb E}^{c}$ be given. Then there exists a $c$-dimensional affine subspace $H\subseteq {\mathbb E}^{d}$ such that $(\forall x\in S)\ \Phi(x)\cap H\neq \emptyset$.
\end{proposition}

\begin{proof}
By Radon's theorem we decompose $S$ into an $r$-point set $R$ and a $b$-point set $B$ (i.e. $R\cap B=\emptyset$ and $R\cup B=S$) such that $\conv(R)\cap \conv(B)\neq \emptyset$. Let $x_{0}\in \conv(R)\cap \conv(B)$ and (recall that by Proposition \ref{injectivity-on-points-prop}  $\Phi (x_0) \ne \emptyset$) let $y_{0}\in \Phi(x_{0})$ be arbitrarily chosen. Since by finiteness of $R$
\begin{equation*}
y_{0}\in \Phi(\conv(R)) = \bigvee_{x\in R} \Phi(x)
\end{equation*}
we obtain that there exists an $(r-1)$-dimensional affine space $H_{R}\subseteq {\mathbb E}^{d}$ such that $y_{0}\in H_{R}$ and $(\forall x\in R)\ H_{R}\cap \Phi(x)\neq \emptyset$. Similarly we obtain a $(b-1)$-dimensional affine space $H_{B}\subseteq {\mathbb E}^{d}$ such that $y_{0}\in H_{B}$ and $(\forall x \in B)\ H_{B}\cap \Phi(x)\neq \emptyset$. Since $y_{0}\in H_{R}\cap H_{B}$ the affine hull $H_{0}$ of $H_{R}\cup H_{B}$ is of dimension $\leq r - 1 + s - 1 = c + 2 - 2 = c$. Further $H_{0}$ intersects all sets $\Phi(x)$ for $x\in S$. Since by hypothesis $c \le d$ the affine space $H_{0}$ can be extended to a $c$-dimensional affine space $H\subseteq {\mathbb E}^{d}$. Thus the proposition is proved.
\end{proof}

\begin{remark}
In case that $d=c+1$ and the images of one point sets under $\Phi$ are located on lines emanating from a common point $o\in {\mathbb E}^{d}$ we obtain the following geometric description. 
\end{remark}

\begin{lemma}\label{radon-affine-dependence-lemma}
Let $c\geq 2$ and let $\Phi:\mathscr{C}({\mathbb E}^{c})\to \mathscr{C}({\mathbb E}^{c+1})$ be a non-trivial lattice homomorphism. Suppose that there exists a point $o\in {\mathbb E}^{c}$ and for any $x\in {\mathbb E}^{c}$ a closed ray $r_{x}$ emanating from $o$ such that $x\neq y$ implies $r_{x}\neq r_{y}$ and
\begin{equation*}
(\forall x\in {\mathbb E}^{c})\quad \Phi(x)\subset r_{x}\setminus o.
\end{equation*}
Then
\begin{enumerate}[(i)]
\item there exists a hyperplane $F\subseteq {\mathbb E}^{c+1}$ such that\label{first-case}
\begin{equation*}
o\in F\quad \hbox{ and }\quad (\forall x\in {\mathbb E}^{c})\ \Phi(x)\subseteq F
\end{equation*}
\item or for any $(c+2)$-point set $S$ there exists a hyperplane $F_{S}\subseteq {\mathbb E}^{c+1}$ such that\label{second-case}
\begin{equation*}
o\notin F_{S}\quad \hbox{ and }\quad (\forall x\in S)\ \Phi(x)\cap F_{S}\neq\emptyset.
\end{equation*}
\end{enumerate}
If (\ref{first-case}) holds, then the sets $\Phi(x)$ are for any $x\in {\mathbb E}^{c}$ one-point sets.
\end{lemma}

\begin{proof} Given a set $S\subset {\mathbb E}^{c}$ we let 
\begin{equation*}
{\cal F}_{S}:=\{ H \subset {\mathbb E}^{c+1} \mid H \hbox{ is a hyperplane and } (\forall x\in S) (\Phi(x)\cap H \neq \emptyset) \}.
\end{equation*}
We already know from Proposition \ref{radon-affine-dependence-proposition} that for any  $(c+2)$-point set $S\subset {\mathbb E}^{c}$ (and thus also for any set $S\subset {\mathbb E}^{c}$ consisting of viewer than $(c+2)$-points) ${\cal F}_{S}\neq \emptyset$.
Thus if (\ref{second-case}) is not fulfilled we obtain a $(c+2)$-point set $S\subseteq {\mathbb E}^{c}$ such that ${\cal F}_{S}\neq \emptyset$ and $(\forall H\in {\cal F}_{S})\ o\in H$. Hence the family 
\begin{equation*}
{\cal X} := \{ S\subseteq {\mathbb E}^{c} \mid S \hbox{ consists of $\leq c+2$ points, ${\cal F}_{S}\neq \emptyset$ and $(\forall H\in {\cal F}_{S})\ o\in H$} \}
\end{equation*}
is not empty and moreover its $(c+2)$-point elements are precisely the $(c+2)$-points sets $S\subseteq {\mathbb E}^{c}$ that do not fulfil (\ref{second-case}).\\ \\
We further obtain (remind that the sets $\Phi(x)$ are by hypothesis for any $x\in {\mathbb E}^{c}$ located on rays emanating from $o$) that 
\begin{equation*}
(\forall S\in {\cal X}) (\forall x\in S)(\forall H\in {\cal F}_{S})\ \Phi(x)\subset H.
\end{equation*}
We thus get for arbitrary $S\in {\cal X}$ letting $A_{S}:=\bigcap_{H\in {\cal F}_{S}} H$ that
\begin{equation}\label{forall x-in-set-implication}
(\forall x\in S)\ (\Phi(x)\subseteq A_{S} \hbox{ and } o\in A_{S}).
\end{equation}
We further remark that the sets $A_{S}$ are (for $S\in {\cal X}$) at most $c$-dimensional affine subspace of ${\mathbb E}^{c+1}$.\\ \\
Note that $S\mapsto \dim(A_{S})$ defines a function from ${\cal X}$ to ${\mathbb R}$. Let $S_{\max}\in {\cal X}$ be such that $S\mapsto \dim(A_{S})$ attains its maximum on ${\cal X}$ at $S_{\max}\in {\cal X}$. Further---by monotonicity of $S\mapsto dim(A_{S})$---we suppose without loss of generality that the set $S_{\max}$ consists of $(c+2)$-points. Thus $S_{\max}$ is a $(c+2)$-point subsets of ${\mathbb E}^{c}$ that does not fulfil (ii) and maximizes the function $S\mapsto dim(A_{S})$ on ${\cal X}$.\\ \\
Let $S_{0}\subset S_{\max}$ be a $(c+1)$-element subset of $S_{\max}$ such that $A_{S_{0}}=A_{S_{\max}}$.\footnote{Such a set $S_{0}$ always exists: The affine space $A_{S_{\max}}$ is at most $c$-dimensional. Thus there exists a $(c+1)$-point set $Q\subseteq\bigcup_{x\in S_{\max}} \Phi(x)$ such that $A_{S_{\max}}=\aff(Q)$. Let $S_{0}$ be an arbitrary $(c+1)$-point subset of $S_{\max}$ such that $Q\subseteq\bigcup_{x\in S_{0}} \Phi(x)$.}
Then for arbitrary $y\in {\mathbb E}^{c}\setminus S_{0}$ we have 
\begin{equation}\label{inclusions-of-the-A_S}
o\in A_{S_{\max}}=A_{S_{0}}\subseteq A_{S_{0} \cup \{ y\} }
\end{equation}
i.e. $o\in A_{S_{0} \cup \{ y\} }$ which implies that the $(c+2)$-point set $S_{0}\cup \{ y \}$ does not fulfil (ii). This further implies by maximality of the dimension of $A_{S_{\max}}$ and (\ref{inclusions-of-the-A_S}) that 
\begin{equation*}
dim(A_{S_{\max}}) = dim(A_{S_{0}}) \leq \dim( A_{S_{0}\cup \{ y\} }) \leq dim(A_{S_{\max}})
\end{equation*}
which by application of $(\ref{inclusions-of-the-A_S})$ and the fact that the sets $A_{S}$ are affine spaces implies that $A_{S_{\max}}= A_{S_{0} \cup \{ y\} }$.\\ \\
From $A_{S_{\max}}= A_{S_{0} \cup \{ y\} }$ and application of \ref{forall x-in-set-implication} with $S=S_{0}\cup \{ y \}$ we obtain that $\Phi(y)\subseteq A_{S_{\max}}$. By the arbitrary choice of $y\in {\mathbb E}^{c}$ we thus obtain 
\begin{equation}\label{Phi-y-subset-A_S...}
(\forall y\in {\mathbb E}^{c}\setminus S_{0})\  (\Phi(y)\subseteq A_{S_{\max}})
\end{equation}
Letting $S=S_{0}$ in (\ref{forall x-in-set-implication}) we obtain from (\ref{forall x-in-set-implication}) and (\ref{Phi-y-subset-A_S...}) that
\begin{equation*} 
(\forall y\in {\mathbb E}^{c})\  (\Phi(y)\subseteq A_{S_{\max}})
\end{equation*}
which together with $dim(A_{S})\leq c$ proves (\ref{first-case}).\\ \\
That in case (\ref{first-case}) the sets $\Phi(x)$ have to be one-point sets follows by an application of Proposition \ref{range-of-Phi-determined-by-image-of-one-point-sets} and the Dimension Lemma \ref{Dimlemma} since $\dim({\mathbb E}^{c})=\dim(F)$ implies that $0$-dimensional convex sets have to be mapped to $0$-dimensional sets (i.e. one-point sets are mapped to one-point sets). 
\end{proof}

\begin{remark}
The following is an analogue of Lemma \ref{radon-affine-dependence-lemma} for segments located on parallel lines (instead of rays emanating from some common point).
\end{remark}

\begin{lemma}\label{radon-affine-dependence-lemma-parallel}
Let $c\geq 2$ and let $\Phi:\mathscr{C}({\mathbb E}^{c})\to \mathscr{C}({\mathbb E}^{c+1})$ be a non-trivial lattice homomorphism and let $g\subset {\mathbb E}^{c+1}$ be a line. Suppose that for any $x\in {\mathbb E}^{c}$ there exists a line $g_{x}$ parallel to $g$ such that $x\neq y$ implies $g_{x}\neq g_{y}$ and
\begin{equation*}
(\forall x\in {\mathbb E}^{c})\quad \Phi(x)\subset g_{x}.
\end{equation*}
Then either
\begin{enumerate}[(i)]
\item there exists a hyperplane $F\subseteq {\mathbb E}^{c+1}$ such that\label{first-case}
\begin{equation*}
F\parallel g\quad \hbox{ and }\quad (\forall x\in {\mathbb E}^{c})\ \Phi(x)\subseteq F
\end{equation*}
\item or for any $(c+2)$-point set $S$ there exists a hyperplane $F_{S}\subseteq {\mathbb E}^{c+1}$ such that\label{second-case}
\begin{equation*}
F_{S}\nparallel g\quad \hbox{ and }\quad (\forall x\in S)\ \Phi(x)\cap F_{S}\neq\emptyset.
\end{equation*}
\end{enumerate}
If (\ref{first-case}) holds, then the sets $\Phi(x)$ are for any $x\in {\mathbb E}^{c}$ one-point sets.
\end{lemma}

\begin{proof}
The lemma is proved completely analogous to Lemma \ref{radon-affine-dependence-lemma}. One just has to replace:\\ \\
\begin{tabular}{lll}
'$o\in H$` &\quad by &\quad '$g \parallel H$`\\
'$o\in A_{S}$` &\quad by &\quad '$g \parallel A_{S}$`\\
'$o\in A_{S_{\max}}$` &\quad by &\quad '$g \parallel A_{S_{\max}}$`\\
'$o\in A_{S_{0} \cup \{ y\} }$` &\quad by &\quad '$g \parallel A_{S_{0} \cup \{ y\} }$`\\
'located on rays emanating from $o$` &\quad by &\quad 'located on lines parallel to $g$`
\end{tabular}
$ $\\ \\
in the proof of Lemma \ref{radon-affine-dependence-lemma}.
\end{proof}

\section{The main Theorem}\label{section-main-theo}

\begin{remark}
Before stating our main result, Theorem \ref{Theorem Homomorphisms}, we investigate in Proposition \ref{Prop Bedingungen für Phi} some hypotheses ensuring that $\Phi: \mathscr{C}(\mathbb{E}^{c}) \to \mathscr{C}(\mathbb{E}^{d})$ fulfils $\Phi(C)=\bigcup_{x \in C} \Phi(x)\cup \Phi(\emptyset)$ and moreover that $\Phi$ is a lattice homomorphism. Together with Theorem \ref{Theorem Homomorphisms} this gives a further characterisation of lattice homomorphisms summarised in Corollary \ref{char-of-homo-corollary}
\end{remark}

\begin{proposition} \label{Prop Bedingungen für Phi}
Let $\Phi: \mathscr{C}(\mathbb{E}^{c}) \to \mathscr{C}(\mathbb{E}^{d})$ be a mapping such that for \mbox{$C\in \mathscr{C}(\mathbb{E}^{c})$} and $x,x_{1},\dots ,x_{n} \in \mathbb{E}^{c}$
\begin{enumerate} [(i)]
\item $x \in C \Rightarrow \Phi(x) \subseteq \Phi(C)$  \label{Bedingung i}
\item $x \notin C \Rightarrow \Phi(x) \cap \Phi(C) = \Phi(\emptyset)$ \label{Bedingung ii}
\item $\bigcup_{x \in \mathbb{E}^{c}} \Phi(x)$ is convex \label{Bedingung iii}
\item $C \subseteq x_1 \vee \dots \vee x_n \Rightarrow \Phi(C) \subseteq \Phi(x_1) \vee \dots \vee \Phi(x_n)$ \label{Bedingung iv}
\end{enumerate}
Then 
\begin{equation}\label{Phi(C)=union}
\forall C \in \mathscr{C}(\mathbb{E}^{c})\ \hbox{ we have }\ \Phi(C)=\bigcup_{x \in C} \Phi(x)\cup \Phi(\emptyset)
\end{equation}
and
\begin{equation}\label{Phi(C)=union-without-empty}
\forall C \in \mathscr{C}(\mathbb{E}^{c})\setminus \{ \emptyset \}\ \hbox{ we have }\ \Phi(C)=\bigcup_{x \in C} \Phi(x).
\end{equation}
Further $\Phi$ is a lattice homomorphism.
\end{proposition}

\begin{remark}\label{Bedingung ii für punkte rem}
Note that hypothesis (\ref{Bedingung ii}) in Proposition \ref{Prop Bedingungen für Phi} implies especially for $x,y\in {\mathbb E}^{c}$ that $x \neq y \Rightarrow \Phi(x) \cap \Phi(y) = \Phi(\emptyset)$.
\end{remark}

\begin{proof}[Proof of Proposition \ref{Prop Bedingungen für Phi}]
Let $C\in \mathscr{C}(\mathbb{E}^{c})$ be arbitrary. From (\ref{Bedingung i}) and (\ref{Bedingung ii}) we obtain that  $\bigcup_{x \in C} \Phi(x)\cup \Phi(\emptyset) \subseteq \Phi(C)$. To prove (\ref{Phi(C)=union}) it thus remains to show that
\begin{equation}\label{Phi-C-Teilmenge-von}
\Phi(C) \subseteq \bigcup_{x \in C} \Phi(x)\cup \Phi(\emptyset).
\end{equation}
Since $C$ is compact we are able to choose points $x_{1},\dots , x_{n}\in {\mathbb E}^{c}$ such that $C \subseteq x_1 \vee  \dots \vee x_n$. By (\ref{Bedingung iii}) and (\ref{Bedingung iv}) we thus obtain that
\begin{equation} \label{Phi C subset Vereinigung Phi x}
\Phi(C) \stackrel{(\ref{Bedingung iv})}{\subseteq} \Phi(x_1) \vee \dots \vee \Phi(x_n) \stackrel{(\ref{Bedingung iii})}{\subseteq} \bigcup_{x \in \mathbb{E}^{c}} \Phi(x).
\end{equation}
Let $y \in \Phi(C) \setminus \Phi(\emptyset)$ be arbitrary. We obtain from (\ref{Phi C subset Vereinigung Phi x}) the existence of some $x \in \mathbb{E}^{c}$ such that $y \in \Phi(x) \setminus \Phi(\emptyset)$, thus by (\ref{Bedingung ii}) that $x \in C$ and thus further $y \in \bigcup_{x \in C} \Phi(x)$. Thus since $y \in \Phi(C) \setminus \Phi(\emptyset)$ was arbitrarily chosen we conclude (\ref{Phi-C-Teilmenge-von}) and thus (\ref{Phi(C)=union}) has been shown. Note that (\ref{Phi(C)=union-without-empty}) is an immediate consequence of (\ref{Phi(C)=union}) since (\ref{Phi(C)=union}) implies $\Phi(\emptyset)\subseteq \Phi(C)$ for any $C\in {\mathscr C}({\mathbb E}^c)$ and thus especially $\Phi(\emptyset)\subseteq \Phi(x)$ for any $x\in {\mathbb E}^{c}$.\\ \\
It remains to show that $\Phi$ is a lattice homomorphism. That $\Phi$ preserves the operation $\wedge$ follows since
\begin{equation*}
\begin{split}
\Phi(C) & \wedge \Phi(D) \stackrel{(\ref{Phi(C)=union})}{=}\left ( \bigcup_{x\in C} \Phi(x) \cup \Phi(\emptyset)\right ) \cap \left ( \bigcup_{y\in D} \Phi(y) \cup \Phi(\emptyset)\right )\\
& = \left ( \bigcup_{x\in C} \bigcup_{y\in D} \Phi(x)\cap \Phi(y) \right ) \cup \Phi(\emptyset) \stackrel{(*)}{=} \bigcup_{x\in C\cap D} \Phi(x) \cup \Phi(\emptyset) \stackrel{(\ref{Phi(C)=union})}{=} \Phi(C\wedge D)
\end{split}
\end{equation*}
with $(*)$ a consequence of Remark \ref{Bedingung ii für punkte rem}. Further we obtain from \ref{Phi(C)=union} that $\Phi$ preserves inclusions and thus we obtain by the convexity of $\Phi(C\vee D)$ that 
\begin{equation*}
\Phi(C)\vee \Phi(D) \subseteq \Phi(C\vee D).
\end{equation*}
It thus remains to show that 
\begin{equation}\label{phi-vee-phi-subset}
\Phi(C\vee D) \subseteq \Phi(C)\vee \Phi(D).
\end{equation}
Let $x\in C\vee D$ be arbitrary. Then there exist $p_{x}\in C$ and $q_{x}\in D$ such that $\{ x \}\subseteq \{ p_{x} \} \vee \{ q_{x} \}$, thus by (\ref{Bedingung iv}) $\Phi(x)\subseteq \Phi(p_{x})\vee \Phi(q_{x})$ and thus further  
\begin{equation*}
\begin{split}
\Phi(C\vee D) & \stackrel{(\ref{Phi(C)=union})}{=}\bigcup_{x\in C\vee D}\Phi(x) \cup \Phi(\emptyset) \subseteq \bigcup_{x\in C\vee D} (\Phi(p_{x})\vee \Phi(q_{x})) \cup \Phi(\emptyset)\\
& \subseteq \bigcup_{p\in C} \bigcup_{q\in D} (\Phi(p)\vee \Phi(q)) \cup \Phi(\emptyset) \stackrel{(\ref{Bedingung i})+(\ref{Phi(C)=union})}{\subseteq} \Phi(C)\vee \Phi(D)
\end{split}
\end{equation*}
i.e. (\ref{phi-vee-phi-subset}) has been shown and the proof is complete.
\end{proof}

\begin{theorem}\label{Theorem Homomorphisms}
Let $c$ be a natural number $\geq 3$ and let $\Phi: \mathscr{C}(\mathbb{E}^{c}) \longrightarrow \mathscr{C}(\mathbb{E}^{c+1})$. Then equivalent are:
\begin{enumerate}[(A)]
\item $\Phi$ is a non-trivial lattice homomorphism.\label{Phi-is-homo} 
\item \label{disjunction-into-4-cases} There exist a hyperplane $H \subset \mathbb{E}^{c+1}$ and an affine bijection $\phi : \mathbb{E}^{c} \to H$ such that precisely one of the following cases holds:
\end{enumerate}
\begin{enumerate} [\ \ \ \ (i)\ ]
\item For any $x\in {\mathbb E}^{c}$ and any $C\in {\cal C}({\mathbb E}^{c})$
\begin{equation*}
\Phi(x)=\phi(x)\ \hbox{ and }\ \Phi(C)=\bigcup_{x\in C}\Phi(x)
\end{equation*}
and thus $\Phi(C)=\phi(C)$ for $\phi(C):=\bigcup_{x\in C} \phi(x)$.\label{case i}
\item $\Phi(\emptyset)=\lbrace o \rbrace$ for some $o \in \mathbb{E}^{c+1}\setminus H$,\label{case ii} and for any $x\in {\mathbb E}^{c}$ and any $C \in \mathscr{C}(\mathbb{E}^{c})$
\begin{equation*}
\Phi(x) = [\phi(x), o] \hbox{ and }\ \Phi(C)=\bigcup_{x\in C} \Phi(x)
\end{equation*}
\item There exists some vector $v \nparallel H$ such that $\forall x\in {\mathbb E}^{c}$ and $\forall C \in \mathscr{C}(\mathbb{E}^{c})$
\begin{equation*}
\Phi(x) = [\phi(x),\phi(x)+v]\ \hbox{ and }\ \Phi(C) = \bigcup_{x\in C} \Phi(x)
\end{equation*}\label{case iii}
\item 
There exist $o \in \mathbb{E}^{c+1}\setminus H$ and $\gamma \in (0,1)$ such that $\forall x\in {\mathbb E}^{c}$ and $\forall C \in \mathscr{C}(\mathbb{E}^{c})$
\begin{equation*}
\Phi(x) = [\phi(x),\gamma \phi(x) + (1-\gamma) o] \ \hbox{ and }\ \Phi(C) = \bigcup_{x\in C} \Phi(x).
\end{equation*}\label{case iv}
\end{enumerate}
\end{theorem}

\begin{proof}
The implication $(\ref{disjunction-into-4-cases})\Rightarrow (\ref{Phi-is-homo})$ in Theorem \ref{Theorem Homomorphisms} is trivial. Thus we prove $(\ref{Phi-is-homo}) \Rightarrow (\ref{disjunction-into-4-cases})$. The proof is based on the Lemmas \ref{Gruberfalllemma}, \ref{E-empty-to-point-lemma} and \ref{empty-to-empty-lemma} that are proved---using the results of Section \ref{preliminaries}, Section \ref{dimension-arguments} and Proposition \ref{Prop Bedingungen für Phi}---in the following two sections.\\ \\
Let $\Phi: \mathscr{C}(\mathbb{E}^{c}) \longrightarrow \mathscr{C}(\mathbb{E}^{c+1})$ be a non-trivial lattice homomorphism. We first show that precisely one of the following cases must hold:
\begin{enumerate} [(I)]
\item \label{one-point-to-one-point-item} $\Phi$ maps any one-point set to a one-point set.
\item \label{empty-to-one-point-item} $\Phi$ maps the empty set to some non-empty set. 
\item \label{dimension-gap-item} $\Phi$ maps the empty set to the empty set and some one-point set to some proper line-segment.
\end{enumerate}
The non-triviality of $\Phi$ implies by Proposition \ref{injectivity-on-points-prop} that $(\forall x\in {\mathbb E}^{c})\ \Phi(\emptyset)\subsetneq \Phi(x)$, consequently by injectivity of $x \mapsto \Phi (x)$ we obtain in case (\ref{one-point-to-one-point-item}) that $\Phi(\emptyset)=\emptyset$ and thus the cases (\ref{one-point-to-one-point-item}), (\ref{empty-to-one-point-item}) and (\ref{dimension-gap-item}) are pairwise disjoint i.e. a homomorphism can not fulfil more than one of the cases. By Lemma \ref{Dimlemma} we know that for any non-trivial homomorphism $\Phi: \mathscr{C}(\mathbb{E}^{c}) \to \mathscr{C}(\mathbb{E}^{c+1})$ it holds that 
\begin{equation}\label{0-leq-dim-diff-leq-1}
(\forall C\in \mathscr{C}(\mathbb{E}^{c}))\ \ 
\dim(C) \leq c-2 \implies  \dim(\Phi(C))-\dim(C)\leq 1.
\end{equation}
A convex set $C$ is $(-1)$-dimensional iff $C=\emptyset$, it is $0$-dimensional iff it is a one-point set and it is $1$-dimensional iff it is a proper line-segment. Thus we obtain from (\ref{0-leq-dim-diff-leq-1}) that the empty set has to be mapped to the empty set or some one-point set. Further by Proposition \ref{injectivity-on-points-prop} one point sets can not be mapped to the empty set and thus by (\ref{0-leq-dim-diff-leq-1}) either any one-point set is mapped to a one-point set or there exists a one-point set that is mapped to a proper line-segment. The case that the empty set is mapped to some one-point set is entirely covered by case (\ref{empty-to-one-point-item}) and in case that the empty set is mapped to the empty set the cases (\ref{one-point-to-one-point-item}) and (\ref{dimension-gap-item}) cover all possibilities for $\Phi$ to deal with one point sets. Thus for any homomorphism one of the cases (\ref{one-point-to-one-point-item}), (\ref{empty-to-one-point-item}) or (\ref{dimension-gap-item}) has to hold. Altogether we have shown that for $\Phi$ precisely one of the cases (\ref{one-point-to-one-point-item}) to (\ref{dimension-gap-item}) holds.\\ \\
Each of the cases (\ref{one-point-to-one-point-item}) to (\ref{dimension-gap-item}) is covered by precisely one of the Lemmas \ref{Gruberfalllemma}, \ref{E-empty-to-point-lemma} and \ref{empty-to-empty-lemma}. Lemma \ref{Gruberfalllemma} shows that (\ref{one-point-to-one-point-item}) implies (\ref{disjunction-into-4-cases} \ref{case i}), while Lemma \ref{E-empty-to-point-lemma} proves that (\ref{empty-to-one-point-item}) implies (\ref{disjunction-into-4-cases} \ref{case ii}). Finally we obtain from Lemma \ref{empty-to-empty-lemma} that (\ref{dimension-gap-item}) implies that either (\ref{disjunction-into-4-cases} \ref{case iii}) or (\ref{disjunction-into-4-cases} \ref{case iv}) holds, which completes the proof of the theorem.
\end{proof}

\begin{corollary}\label{char-of-homo-corollary}
Let $c\geq 3$. Then $\Phi: \mathscr{C}(\mathbb{E}^{c}) \to \mathscr{C}(\mathbb{E}^{c+1})$ is a non-trivial lattice homomorphism iff $\Phi$ fulfils the hypotheses (\ref{Bedingung i})-(\ref{Bedingung iv}) of Proposition \ref{Prop Bedingungen für Phi} and is in addition non-trivial i.e. there exist $C,D\in \mathscr{C}(\mathbb{E}^{c})$ such that $\Phi(C)\neq \Phi(D)$.
\end{corollary}

\begin{proof}
From Theorem \ref{Theorem Homomorphisms} one easily derives that any non-trivial lattice homomorphism $\Phi: \mathscr{C}(\mathbb{E}^{c}) \to \mathscr{C}(\mathbb{E}^{c+1})$ fulfils the hypotheses of Proposition \ref{Prop Bedingungen für Phi}. Conversely we obtain from Proposition \ref{Prop Bedingungen für Phi} that the mapping $\Phi$ is a lattice homomorphism and thus since $\Phi$ is non-trivial a non-trivial lattice homomorphism.
\end{proof}

\begin{remark}
From the proof and statement of Theorem \ref{Theorem Homomorphisms} it becomes clear that:\\ \\
Case (\ref{disjunction-into-4-cases} \ref{case i}) in Theorem \ref{Theorem Homomorphisms} applies iff $\Phi$ maps one-point sets to one-point sets which is the case iff $\Phi$ keeps the dimension of any convex body constant i.e. 
\begin{equation*}
(\forall C\in\mathscr{C}(\mathbb{E}^{c}))\ \ \dim \Phi(C)=\dim C.
\end{equation*}
Case (\ref{disjunction-into-4-cases} \ref{case ii}) in Theorem \ref{Theorem Homomorphisms} applies iff $\Phi$ maps the empty set to some non-empty set iff $\Phi$ maps the empty set to some one-point set iff $\Phi$ rises the dimension of any convex body by $1$ i.e. 
\begin{equation*}
(\forall C\in\mathscr{C}(\mathbb{E}^{c}))\ \ \dim \Phi(C)=\dim C + 1.
\end{equation*}
Case (\ref{disjunction-into-4-cases} \ref{case iii}) and (\ref{case iv}) in Theorem \ref{Theorem Homomorphisms} apply iff $\Phi(\emptyset)=\emptyset$ and $\Phi$ maps some point to some convex body of dimension $\geq 1$ iff $\Phi(\emptyset)=\emptyset$ and $\Phi$ maps some point to a proper line-segment iff $\Phi$ raises the dimension of any non-empty convex body by $1$ and keeps the dimension of the empty set constant i.e. 
\begin{equation*}
(\forall C\in\mathscr{C}(\mathbb{E}^{c})\setminus \{ \emptyset \})\ \ \dim (\Phi(C))=\dim (C) + 1\ \hbox{ and }\ \dim \Phi(\emptyset) = \dim \emptyset = (-1).
\end{equation*}
\end{remark}

\begin{remark}
The image of $\mathscr{C}(\mathbb{E}^{c})$ under the mapping $[\dim\circ \Phi]:{\mathscr C}({\mathbb E}^{c})\to {\mathbb Z}$ is\\ \\
$\{ -1, 0, 1,\dots c\}$ in case (\ref{disjunction-into-4-cases} \ref{case i}),\\ \\
$\{ 0, 1,\dots c+1\}$ in case (\ref{disjunction-into-4-cases} \ref{case ii}) and\\ \\
$\{ -1, 1,\dots c+1\}$ in the cases (\ref{disjunction-into-4-cases} \ref{case iii}) and (\ref{case iv}) of Theorem \ref{Theorem Homomorphisms}.\\ \\
Thus in the cases (\ref{disjunction-into-4-cases} \ref{case i}) and (\ref{case ii}) the image $[\dim\circ \Phi](\mathscr{C}(\mathbb{E}^{c}))$ is an order interval in ${\mathbb Z}$ containing $0$, while in the cases (\ref{disjunction-into-4-cases} \ref{case iii}) and (\ref{case iv}) the image possesses a gap at $0$. We call this gap the dimension-gap. The proofs of Lemma \ref{Gruberfalllemma} and Lemma \ref{E-empty-to-point-lemma}---covering the cases without dimension-gap---are considerably simpler than the proof of Lemma \ref{empty-to-empty-lemma} that covers the situations in which a dimension gap occurs.   
\end{remark}

\section{Cases without dimension-gap}\label{section-no-dim-gape}

\begin{lemma} \label{Gruberfalllemma}
Let $\Phi: \mathscr{C}(\mathbb{E}^{c}) \to \mathscr{C}(\mathbb{E}^{c+1})$ be a non-trivial lattice homomorphism, let $c \geq 2$ and suppose that $\Phi$ maps one-point sets to one-point sets. Then $\Phi(\emptyset)=\emptyset$ and there exists a hyperplane $H\subset \mathbb{E}^{c+1}$ and an affine bijection $\phi:{\mathbb E}^{c}\to H$ such that $\Phi(x) = \phi(x)$ for $x \in \mathbb{E}^{c}$. Further $\Phi(C)=\bigcup_{x\in C} \Phi(x)$ for any $C\in \mathscr{C}({\mathbb E}^{c})$.
\end{lemma}

\begin{proof}
That $\Phi(\emptyset)=\emptyset$ is a consequence of the hypotheses and Proposition \ref{injectivity-on-points-prop}. Define the mapping $\phi$ by $\{ \phi(x) \} = \Phi(x)$. Then by Proposition \ref{injectivity-on-points-prop} $\phi(x)$ is injective. Further $z\in (x,y)$ implies that 
\begin{equation*}
\phi(z)\in \phi(x)\vee \phi(y)\setminus \{\phi(x),\phi(y) \}=(\phi(x),\phi(y))
\end{equation*}
i.e. $\phi$ preserves the order of points and thus is by Corollary \ref{corollary-aff-geometry} an affine bijection onto its image the hyperplane $H:=\{ \phi(x) \mid x\in {\mathbb E}^{c} \}$. Thus
\begin{equation*}
\bigcup_{x \in \mathbb{E}^{c}} \Phi(x) = \bigcup_{x \in \mathbb{E}^{c}} \phi(x) = \phi(\mathbb{E}^{c}) = H\ \hbox{ is convex}
\end{equation*}
i.e. hypothesis (\ref{Bedingung iii}) of Proposition \ref{Prop Bedingungen für Phi} is fulfilled. Since $\Phi$ is a lattice homomorphism one easily shows that the hypotheses (i), (ii) and (iv) of Proposition \ref{Prop Bedingungen für Phi} are equally fulfilled. By application of Proposition \ref{Prop Bedingungen für Phi} and the fact that $\Phi(\emptyset)=\emptyset$ we obtain that $\Phi(C)=\bigcup_{x\in C} \Phi(x)= \bigcup_{x\in C}\phi(x) =: \phi(C)$ for any $C\in \mathscr{C}({\mathbb E}^{c})$.
\end{proof}

\begin{lemma}\label{E-empty-to-point-lemma}
Let $\Phi: \mathscr{C}(\mathbb{E}^{c}) \to \mathscr{C}(\mathbb{E}^{c+1})$ be a non-trivial lattice homomorphism let $c \geq 3$ and suppose that $\Phi$ maps the empty set to some non-empty set. Then $\Phi(\emptyset)=\{ o \}$ for some $o\in {\mathbb E}^{c+1}$ and $\Phi$ maps any one-point set to a proper line-segment containing $o$ as one of its endpoints. Further there exists a hyperplane $H\subset \mathbb{E}^{c+1}$ with $o\notin H$ and an affine bijection $\phi:{\mathbb E}^{c}\to H$ such that for any $x\in {\mathbb E}^{c}$ and any $C \in \mathscr{C}(\mathbb{E}^{c})$
\begin{equation}\label{case-empty-to-point-result}
\Phi(x) = [\phi(x), o] \hbox{ and }\ \Phi(C)=\bigcup_{x\in C} \Phi(x).
\end{equation}
\end{lemma}

\begin{proof}
By the Dimension Lemma \ref{Dimlemma} 
and the hypothesis that $\Phi(\emptyset)\neq \emptyset$ we obtain that $\Phi(\emptyset)$ is a zero-dimensional convex set i.e. a one-point set i.e. $\Phi(\emptyset) =\{ o \}$ for some $o\in {\mathbb E}^{d}$. Further by Proposition \ref{injectivity-on-points-prop} $\Phi(x) \supsetneq \Phi(\emptyset)=\{ o \}$ and thus---taking the Dimension Lemma \ref{Dimlemma} into account---the convex set $\Phi(x)$ is for any $x \in \mathbb{E}^{c}$ one dimensional and thus a proper line-segment. Thus for $x,y\in {\mathbb E}^{c}$ we have 
\begin{equation}\label{intersection-is-o}
x \neq y\ \hbox{ implies }\ \Phi(x) \wedge \Phi(y) = \Phi(\emptyset) = \lbrace o \rbrace\ \hbox{ and }\ \Phi(x) \neq \Phi(y).
\end{equation}
We show next that $o$ has to be an endpoint of the proper line-segment $\Phi(x)$ for any $x\in \mathbb{E}^{c}$. Indirect: Suppose $o \in \relint(\Phi(x))$ and let $y, z$ be points on a line through $x$ such that $y \in  (x,z)$. Then 
\begin{equation}\label{Phi_x_y_z_subset_Phi_x_x}
\Phi(x),\ \Phi(y),\ \Phi(z)\subseteq \Phi(x\vee z).
\end{equation}
Since by application of the Dimension Lemma \ref{Dimlemma} (note that here the hypothesis $c\geq 3$ enters)
\begin{equation*}
\dim(\Phi(x\vee z))\leq \dim(x\vee z) + 1 = 2
\end{equation*}
the segments $\Phi(x)$, $\Phi(y)$, $\Phi(z)$ are contained in a two-dimensional affine subspace $F$ of ${\mathbb E}^{c+1}$ and contain the common point $o$. Denote by $p,q$ the endpoints of $\Phi(x)$, by $r \neq o$ an endpoint of $\Phi(y)$ and denote the endpoints of $\Phi(z)$ by $s,t$. Then by (\ref{Phi_x_y_z_subset_Phi_x_x}) and the fact that $\Phi$ is a lattice-homomorphism
\begin{equation*}
r\in \Phi(y)\subseteq \Phi(x\vee z) = \Phi(x)\vee \Phi(z)=\triangle (p,q,s)\cup \triangle (p,q,t).
\end{equation*} 
We suppose without loss of generality that $r\in \triangle(p,q,s)$.
Note that by (\ref{intersection-is-o}) we have that $r\notin \aff([p,q])$ and thus that the triangle $\triangle (p,q,s)$ is non degenerate. Thus, since $o \in (p,q)$, we obtain by application of Proposition \ref{Dreiecksprop} 
\begin{equation*}
\begin{split}
\emptyset \neq \relint([s,o] \cap \triangle (r,p,q))\setminus \{ o \} \subseteq (\Phi(z)\wedge \Phi(x\vee y))\setminus \{ o \} & \\
=\Phi(z\wedge (x\vee y))\setminus \{ o \} = \Phi(\emptyset)\setminus \{ o \} =\{ o \}\setminus \{ o \} & = \emptyset.
\end{split}
\end{equation*}
\emph{Contradiction}. Hence $\Phi$ maps any one-point set to a proper line-segment containing $o$ as one of its endpoints.\\ \\
We can thus define a function $\phi: \mathbb{E}^{c} \mapsto \mathbb{E}^{c+1}$ that maps any $x \in \mathbb{E}^{c}$ to the unique endpoint of $\Phi(x)$ that differs from $o$ i.e. $\phi(x)$ is implicitly given by $[\phi(x),o]:=\Phi(x)$. Note that $\phi$ is by (\ref{intersection-is-o}) injective.\\ \\
We show next that $\phi$ preserves the order of points i.e. we show that 
\begin{equation}\label{convex-in-convex}
x \in (s_1,s_2)\ \hbox{ implies }\ \phi(x) \in (\phi(s_1),\phi(s_2))
\end{equation}
Let $x, s_1, s_2 \in \mathbb{E}^{c}$ be such that $x \in (s_1,s_2)$ and let $u_1,v_{1} \in \mathbb{E}^{c}$ be points such that $s_1 \in (u_1,v_{1})$ and $(u_{1},v_{1}) \nparallel (s_{1}.s_{2})$. By Proposition \ref{existence-of-second-line} we obtain $u_{2},v_2 \in \mathbb{E}^{c}$ such that $s_{2}\in (u_{2},v_{2})$ and $x,s_1,s_2,u_1,v_{1},u_{2},v_2$ are located in one and the same plane such that
\begin{equation}\label{x=cap}
\begin{split}
\{ x \} = (u_{1},v_{2})\cap (v_{1}, u_{2})=(u_{1},v_{2})\cap (s_{1},s_{2})=(v_{1},u_{2})\cap (s_{1},s_{2})\ \hbox{and}\\[0.1cm]
\{ x \} = [u_{1},v_{2}]\cap [v_{1}, u_{2}]=[u_{1},v_{2}]\cap [s_{1},s_{2}]=[v_{1},u_{2}]\cap [s_{1},s_{2}].\ \hbox{\ \ }
\end{split}
\end{equation}
Since $\Phi$ is a homomorphism we obtain from (\ref{x=cap}) that
\begin{equation}\label{E(x)=cap}
\Phi(x)=\Phi(u_1 \vee v_2) \wedge \Phi(s_{1}\vee s_{2}).
\end{equation}
Note that by (\ref{x=cap}), (\ref{intersection-is-o}) and the definition of $\phi$ we obtain that
\begin{equation}\label{f(x)-notin}
\begin{split}
\phi(x) \notin \Phi(u_{1})\cup \Phi(v_{2}) = [o,\phi(u_{1})]\cup [o,\phi(v_{2})]\ \hbox{ and }\\[0.1cm]
\phi(x) \notin \Phi(s_{1}) \cup \Phi(s_{2}) = [o,\phi(s_{1})]\cup [o,\phi(s_{2})].\ \hbox{\ \ }
\end{split}
\end{equation}
Note further that by definition of $\phi$
\begin{equation}\label{conv=triangle}
\begin{split}
\Phi(u_1 \vee v_2) = \Phi(u_1) \vee \Phi(v_2) = [\phi(u_{1}),o]\vee [\phi(v_{2}),o] = \triangle (o, \phi(u_1), \phi(v_2))\ \hbox{ and }\\[0.1cm]
\Phi(s_{1}\vee s_{2}) = \Phi(s_1) \vee \Phi(s_2) = [\phi(s_{1}),o]\vee [\phi(s_2),o] = \triangle (o, \phi(s_1), \phi(s_2)).\ \hbox{\ \ }
\end{split}
\end{equation}
From (\ref{E(x)=cap}) and (\ref{conv=triangle}) we obtain
\begin{equation}\label{f(x)-o-is-intersection}
\begin{split}
[\phi(x),o] & = \Phi(x)=\Phi(u_1 \vee v_2) \wedge \Phi(s_{1}\vee s_{2})\\[0.1cm] 
& =  \triangle(o,\phi(u_{1}),\phi(v_{2})) \cap \triangle(o,\phi(s_{1}),\phi(s_{2})).
\end{split}
\end{equation}
We proceed indirectly: Suppose that $\phi(x) \notin [\phi(s_1),\phi(s_2)]$ i.e that (\ref{convex-in-convex}) were not fulfilled.
Then by (\ref{f(x)-o-is-intersection}) and the second line in (\ref{f(x)-notin})
\begin{equation}\label{f(x)-notin-relint}
\phi(x) \in \relint(\triangle (o, \phi(s_1), \phi(s_2))).
\end{equation}
From (\ref{f(x)-notin-relint}), (\ref{f(x)-o-is-intersection}) and the first line of (\ref{f(x)-notin}) we obtain by application of Proposition \ref{relint-two-triangles-intersection-prop} with $y=\phi(x), p=\phi(s_1), q=\phi(s_2), r=\phi(u_1)$ and $s=\phi(v_2)$ that 
\begin{equation}\label{f(x)-in-(f(u_1),f(v_2)}
\phi(x)\in (\phi(u_{1}),\phi(v_{2})).
\end{equation}
Analogously one shows that
\begin{equation}\label{f(x)-in-(f(u_1),f(v_2)-analogon}
\phi(x)\in (\phi(v_{1}),\phi(u_{2})). 
\end{equation}
Since by (\ref{f(x)-in-(f(u_1),f(v_2)}) and (\ref{f(x)-in-(f(u_1),f(v_2)-analogon}) the non-degenerate line-segments $(\phi(u_{1}),\phi(v_{2}))$ and $(\phi(v_{1}),\phi(u_{2}))$ intersect in the common point $\phi(x)$ we obtain that\footnote{In fact the affine space $F$ defined in $(*)$ has to be a plane or a line. If $F$ were a line we would have, e.g., $\phi (v_1) \in [\phi (u_1), \phi (v_2)]$, which implies
$[o, \phi (v_1)] \subset \Delta \left( o, \phi (u_1), \phi (v_2) \right) 
$. We are going to show that $v_1 \in [u_1, v_2]$.\\ \\
Indirect: In fact, else we have $v_1 \land (u_1 \lor v_2) = \emptyset $ and hence 
\begin{equation*}
\begin{split}
[o, \phi (v_1)] \cap \Delta \left( o, \phi (u_1), \phi (v_2) 
\right) = [o, \phi (v_1)] \land \left( [o, \phi (u_1)] \lor [o, \phi (v_2)] 
\right)\\
= \Phi (v_1) \land 
\left( \Phi (u_1) \lor \Phi (v_2) \right) = \Phi ( \emptyset ) = \{ o \}.
\end{split}
\end{equation*}
Therefore on one hand $[o, \phi (v_1)]$ is a proper 
segment in $\Delta \left( o, \phi (u_1), \phi (v_2) \right) $, and on the 
other hand $[o, \phi (v_1)]$ intersects $\Delta \left( o, \phi (u_1), \phi 
(v_2) \right) $ just in $o$. {\it Contradiction.}\\ \\
Thus $v_1 \in 
[u_1,v_2]$ and hence $v_1 \in [u_1,v_2] \cap [v_1,u_2]$ contradicting (\ref{x=cap}), i.e. we proved that the hypothesis that $F$ is a line is contradictory and thus $F$ has to be a plane.}
\\ \\ 
$(*)$\ \ \ \ \ \ \ \parbox{11cm}{$F:= \aff(\phi(v_{1}),\phi(u_{2}),\phi(u_1),\phi(v_2))$ is a plane with $\phi(x)\in F$.}\\ \\
Let $f_{1}, f_{2}, h\in \{ u_{1}, u_{2}, v_{1}, v_{2} \}$ be such that $f_{1}\neq h\neq f_{2}$. Then for $e_{1}=\phi(f_{1})$, $e_{2}=\phi(f_{2})$ and $w=\phi(h)$ we obtain by application of (\ref{intersection-is-o}) and injectivity of $\phi$ that
\begin{equation}\label{e_1-neq w-neq}
e_{1}\neq w\neq e_{2}
\end{equation}
and further calculate that
\begin{equation}\label{triangle(e_1,e_2,o)}
\begin{split}
\triangle(e_{1},e_{2},o)\cap [w,o] = ([e_{1}, o] \vee [e_{2},o]) \wedge [w,o] = (\Phi(f_{1})\vee \Phi(f_{2}))\wedge \Phi(h)\\
= \Phi((f_{1}\vee f_{2})\wedge h) = \Phi(\emptyset) = \{ o \}.
\end{split}
\end{equation}
From (\ref{e_1-neq w-neq}) and (\ref{triangle(e_1,e_2,o)}) we obtain by application of Proposition \ref{quadrangle-prop} with $a_{i}=\phi(u_{i})$, $b_{i}=\phi(v_{i})$ and $z=\phi(x)$ that $o\notin F$.\\ \\
Let $C:=(o\vee \phi(u_{1})\vee \phi(v_{1})\vee \phi(u_{2})\vee \phi(v_{2}))=\Phi(u_{1})\vee \Phi(v_{1})\vee \Phi(u_{2})\vee \Phi(v_{2})$. Then $\phi(s_{i})\in \Phi(s_{i})\subseteq \Phi(u_{i})\vee \Phi(v_{i})\subseteq C$ and therefore $\phi(s_{1}),\phi(s_{2})\in C$. Since $o \in C\setminus F$ and thus $\relint(\triangle (o,\phi(s_{1}),\phi(s_{2})))\subseteq C\setminus F$ we obtain from (\ref{f(x)-notin-relint}) that $\phi(x)\notin F$ contradicting $(*)$. \emph{Contradiction.}\\ \\
So (\ref{convex-in-convex}) has been shown and thus we know that $\phi$ preserves the order of points. Since $\phi$ is injective and $c\geq 2$ we obtain by application of Corollary \ref{corollary-aff-geometry} that $\phi$ is an affine bijection onto its image $\phi({\mathbb E}^{c})=:H$ and thus further that $H$ is an affine subspace of ${\mathbb E}^{c}$ with $\dim(H)=c$ and thus a hyperplane in ${\mathbb E}^{c+1}$. By construction of $\phi$ we have $\Phi(x) = [o,\phi(x)]$. Hence $\bigcup_{x \in \mathbb{E}^{c}} \Phi(x)$ is convex and by application of Proposition \ref{Prop Bedingungen für Phi} we obtain that 
\begin{equation*}
\Phi(C)=\bigcup_{x\in C} \Phi(x)\cup \Phi(\emptyset)=\bigcup_{x\in C} \Phi(x)\cup \{ o \} = \bigcup_{x\in C} \Phi(x).
\end{equation*}
The lemma is proved.
\end{proof}

\section{The case of a dimension-gap}\label{section-dim-gape}

\begin{proposition}\label{more-than-one-line-segment-prop}
Let $\Phi:{\mathscr C}({\mathbb E}^{c})\to {\mathscr C}({\mathbb E}^{c+1})$ be a non-trivial lattice homomorphism with $\Phi(\emptyset)=\emptyset$ and suppose that $c\geq 3$. Let $x,y\in {\mathbb E}^{c}$ be arbitrary (and note that we do not suppose that $x\neq y$). Then there exists a plane $F$ such that
\begin{equation}\label{z-not-in-affine}
F\supset \aff(x,y)\ \hbox{ and }\ \forall z\in {\mathbb E}^{c}\setminus F\quad \Phi(z)\nsubseteq \aff(\Phi(x)\vee \Phi(y))
\end{equation}
In the case that $\Phi(x)$ is a proper line-segment, there exists some $z\in {\mathbb E}^{c}\setminus F$ such that $\Phi(z)$ is additionally a proper line-segment.
\end{proposition}

\begin{proof} 
The statement for $x=y$ follows from the statement for $x \ne y$, so we suppose $x \ne y$. We proceed indirectly. Suppose that formula (\ref{z-not-in-affine}) does not hold. Then we can find a three dimensional simplex $P=x \vee y \vee s \vee t\subseteq {\mathbb E}^{c}$ (note that here the hypothesis $c\geq 3$ enters) such that $\Phi (P) \subseteq \aff ( \Phi (x) \vee \Phi (y))$. Hence $\dim \Phi (P) \leq \dim \Phi (x \vee y) \leq 2$ by Lemma \ref{Dimlemma} and $c \ge 3$. This contradicts Proposition  \ref{dimprop}. Thus formula (\ref{z-not-in-affine}) holds.\\ \\ 
Next we show that if $\Phi(x)$ is a proper line-segment, then we can choose $z\in {\mathbb E}^{c}\setminus F$ such that $\Phi(z)$ is also a proper line-segment. Note that by Lemma \ref{Dimlemma} and Proposition \ref{dimprop} it is clear that $\Phi(z)$ is either a point or a proper line-segment.\\ \\
We again proceed indirectly and suppose that 
\begin{equation}\label{E(w)-notin-affine}
z \in {\mathbb E}^{c}\setminus F\ \hbox{ and }\ \Phi(z)\nsubseteq \aff(\Phi(x)\vee \Phi(y))\ \hbox{ implies that }\ \Phi(z)=\{ p_{z} \}
\end{equation}
for some $p_{z}\in {\mathbb E}^{c+1}$. Since $x\in F$ we have
\begin{equation}\label{w-notin-2x-w-notin}
w \in {\mathbb E}^{c}\setminus F\ \Rightarrow\ (2x-w)\in {\mathbb E}^{c}\setminus F.
\end{equation}
Further
\begin{equation}\label{E(x)-sub-E(w)-E(2x-w)}
(\forall w\in {\mathbb E}^{c})\ \ \Phi(x)\subseteq \Phi(w)\vee \Phi(2x-w)
\end{equation}
and
\begin{equation}\label{Phi(x)_cap_Phi(2x-w)}
(\forall w\in {\mathbb E}^{c}\setminus F)\ \ \Phi(x)\cap \Phi(2x-w) = \Phi(\emptyset) = \emptyset.
\end{equation}
Let $w \in {\mathbb E}^{c}\setminus F$. Then in accordance with (\ref{z-not-in-affine}) 
\begin{equation}\label{E(w)-notin-affine-second}
\Phi(w)\nsubseteq \aff(\Phi(x)\vee \Phi(y))
\end{equation}
In case that $\Phi(w)$ and $\Phi(2x-w)$ were one-point sets we would obtain from the hypothesis that $\Phi(x)$ is a proper line-segment and from (\ref{E(x)-sub-E(w)-E(2x-w)}) that 
\begin{equation*}
\aff(\Phi(x)) = \aff(\Phi(w)\vee \Phi(2x-w))
\end{equation*}
contradicting (\ref{E(w)-notin-affine-second}). Thus---since by (\ref{E(w)-notin-affine-second}) and 
(\ref{E(w)-notin-affine}) we have that $\Phi(w)= \{ p_{w} \}$ is a one-point set---we obtain that $\Phi(2x-w)$ has to be a proper line-segment. From the now established fact that $\Phi(2x-w)$ has to be a proper line-segment, (\ref{w-notin-2x-w-notin}) and hypothesis (\ref{E(w)-notin-affine}) with $z=2x-w$ we obtain that 
\begin{equation}\label{Phi-2x-w-subset-affine}
\Phi(2x-w)\subseteq \aff(\Phi(x)\vee \Phi(y))
\end{equation}
Further since $\Phi(w)=\{ p_{w} \}$ is a one-point set we obtain from (\ref{E(x)-sub-E(w)-E(2x-w)}) and (\ref{Phi(x)_cap_Phi(2x-w)}) that
\begin{equation*}
\Phi(w)\subseteq aff(\Phi(x)\vee \Phi(2x-w))
\end{equation*}
and thus by (\ref{Phi-2x-w-subset-affine}) that $\Phi(w)\subseteq \aff(\Phi(x)\vee \Phi(y))$ contradicting (\ref{E(w)-notin-affine-second}).
\end{proof}

\begin{notation}\label{intervals-to-lines-notation}
Given $\Phi:{\mathscr C}({\mathbb E}^{c})\to {\mathscr C}({\mathbb E}^{c+1})$ we let 
\begin{equation}
G:=\{ x\in {\mathbb E}^{c} \mid \Phi(x)\ \hbox{is a proper line-segment}\}.
\end{equation}
We further let ${\cal G}:=\{ \aff(\Phi(x)) \mid x\in G \}$ and note that the elements of ${\cal G}$ are lines in ${\mathbb E}^{c+1}$.
\end{notation}

\begin{proposition}\label{for-any-two-lines-a-third-prop}
Let $c\geq 3$ and let $\Phi:{\mathscr C}({\mathbb E}^{c})\to {\mathscr C}({\mathbb E}^{c+1})$ be a non-trivial lattice homomorphism with $\Phi(\emptyset)=\emptyset$. Suppose that there exists some $x\in {\mathbb E}^{c}$ such that $\Phi(x)$ is a proper line-segment.
Then for any (not necessarily distinct) lines $g_{1}, g_{2}\in {\cal G}$ there exists a line $g_{0}\in {\cal G}$ such that $g_{0}\nsubseteq \aff(g_{1}\cup g_{2})$ and ${\cal G}$ consists  
of at least three distinct lines $g_{0},g_{1},g_{2}$ such that $g_{0}\nsubseteq \aff(g_{1}\cup g_{2})$.
\end{proposition}

\begin{proof}
This is a consequence of Proposition \ref{more-than-one-line-segment-prop} (and Notation \ref{intervals-to-lines-notation}).
\end{proof}

\begin{lemma}\label{empty-to-empty-lemma}
Let $\Phi: \mathscr{C}(\mathbb{E}^{c}) \longrightarrow \mathscr{C}(\mathbb{E}^{c+1})$ be a non-trivial lattice homomorphism, let $c \geq 3$ and suppose that $\Phi(\emptyset)=\emptyset$ and $\Phi$ maps some point to some set of dimension $\geq 1$. Then $\Phi$ maps any one-point set to a proper line-segment such that for arbitrary $x\neq y\in {\mathbb E}^{c}$ we have $\Phi(x)\wedge \Phi(y)=\emptyset$ and precisely one of the following two cases holds:\\ \\
a) All line-segments $\Phi(x)$ are contained in different rays emanating from one common point $o$ and there exists a hyperplane $H\subset \mathbb{E}^{c+1}$ with $o\notin H$ and an affine bijection $\phi:{\mathbb E}^{c}\to H$ and a constant $\gamma \in (0,1)$ such that 
\begin{equation}\label{segments-on-rays-result}
\Phi(x) = [\phi(x),\gamma \phi(x) + (1-\gamma) o] \ \hbox{ and }\ \Phi(C) = \bigcup_{x\in C} \Phi(x).
\end{equation}  
b) All line-segments $\Phi(x)$ are parallel to some vector $v$ and of the same length as $v$ and there exist a hyperplane $H \subset \mathbb{E}^{c+1}, H \nparallel v$ and an affine bijection $\phi:{\mathbb E}^{c}\to H$ such that 
\begin{equation}\label{parallel-result}
\Phi(x) = [\phi(x),\phi(x)+v]\ \hbox{ and }\ \Phi(C) = \bigcup_{x\in C} \Phi(x).
\end{equation} 
\end{lemma}

\begin{proof} 
Let $\Phi$ fulfil the hypotheses of the lemma. Since $\Phi(\emptyset)=\emptyset$ the images of distinct points are disjoint. We know from Lemma \ref{Dimlemma} and Proposition \ref{dimprop} that points can only be mapped to $0$-dimensional or $1$-dimensional convex bodies i.e. to points or proper line-segments. Thus according to our hypotheses some point in ${\mathbb E}^{c}$ is mapped to a proper line-segment. We thus obtain from Proposition \ref{for-any-two-lines-a-third-prop}---in the notation introduced in \ref{intervals-to-lines-notation}---that ${\cal G}$ contains three distinct lines $g_{0},g_{1},g_{2}$ such that $g_{0}\nsubseteq \aff(g_{1},g_{2})$ and we distinguish the following two cases:
\begin{enumerate} [(i)]
\item Any two distinct lines in ${\cal G}$ are not parallel 
\label{non-parallel-case}
\item There exist two distinct parallel lines in ${\cal G}$. 
\label{case a}
\end{enumerate}
We are going to show that (\ref{non-parallel-case}) implies that case a) of the Lemma holds, while (\ref{case a}) implies that case b) holds. Since (\ref{non-parallel-case}) and (\ref{case a}) cover all possible situations we thus obtain that there exist beside a) and b) no further cases.\\ \\
Note that our hypotheses imply that 
\begin{equation}\label{intersection-equals-empty-set}
(\forall q,w  \in {\mathbb E}^{c})\quad q\neq w\ \Rightarrow\ \Phi(q)\cap \Phi(w)=\emptyset
\end{equation}
\begin{enumerate}[(A)]
\item We consider case (\ref{non-parallel-case}) first 
i.e. for distinct lines $g_{1},g_{2}\in {\cal G}$ we have $g_{1}\nparallel g_{2}$.
\begin{enumerate}[\hspace{-0.8em} (1)]
\item\label{I.1} We show that all lines that are elements of ${\cal G}$ intersect in some common point $o$.
Let $g_1, g_2 \in \mathcal{G}$ be distinct lines and let $x,y\in {\mathbb E}^{c}$ be arbitrary points such that 
\begin{equation*}
\aff(\Phi(x))=g_{1},\ \ \aff(\Phi(y))=g_{2}.
\end{equation*}
By application of the Dimension Lemma \ref{Dimlemma} (note that here the hypothesis $ c \geq 3$ enters; compare with Remark \ref{Remark c geq 3}) we obtain that 
\begin{equation*}
\dim (\Phi(x) \vee \Phi(y)) = \dim(\Phi(x\vee y)) \leq \dim(x\vee y) +1 = 2
\end{equation*}
and thus there exists a plane $F$ with $g_{1}, g_{2} \subseteq F$. Thus $g_{1} \nparallel g_{2}$ implies that the lines $g_{1}$ and $g_{2}$ intersect i.e. any two distinct lines $g_{1}, g_{2}\in {\cal G}$ intersect. By Proposition \ref{for-any-two-lines-a-third-prop} there exist three different lines $g_{0}, g_{1}, g_{2}\in {\cal G}$ such that $g_{0}\nsubseteq \aff(g_{1},g_{2})$
and thus by Proposition \ref{Geradenprop} all lines in ${\cal G}$ intersect in some common point $o$. Thus (\ref{I.1}) has been proved.
\item\label{I.2} We show that for any ray $r$ emanating from $o$ there does not exist more than one $x\in {\mathbb E}^{c}$ with $(\Phi(x)\cap r)\setminus \{ o \} \neq \emptyset$.\\ \\
Indirect: Suppose that $x,y\in {\mathbb E}^{c}$ are distinct points such that\linebreak  $(\Phi(x)\cap r)\setminus \{ o \} \neq \emptyset$ and $(\Phi(y)\cap r)\setminus \{ o \} \neq\emptyset$ and note that by disjointness of $\Phi(x)$ and $\Phi(y)$ we obtain $\aff(\Phi(x)\vee \Phi(y)) \supseteq \aff(r)$ and thus further by (\ref{I.1}) $\aff(\Phi(x)\vee \Phi(y)) = \aff(r)$. By Proposition \ref{more-than-one-line-segment-prop} there exists $z\in {\mathbb E}^{c}\setminus \aff(x,y)$ such that
\begin{equation*}
\Phi(z)\nsubseteq \aff(\Phi(x)\vee \Phi(y)) = \aff(r)
\end{equation*}
and $\Phi(z)$ is a non-degenerate line-segment. By (\ref{I.1}) $\Phi(z)$ is contained in some line $g$ through $o$. Thus we obtain the contradiction
\begin{equation*}
\begin{split}
2 \stackrel{(a)}{\leq} \dim [(\Phi(z)\vee \Phi(x)) \wedge (\Phi(z)\vee \Phi(y))] & = \dim [\Phi((x\vee z)\wedge (y\vee z))]\\ 
& \stackrel{(b)}{=} \dim(\Phi(z)) \stackrel{(c)}{=} \dim(z) + 1 =  1
\end{split}
\end{equation*}
with (a) a consequence of Proposition \ref{different-rays-two-dim-prop}, (b) a consequence of $x\neq y$ and $z\notin \aff(x\vee y)$ and (c) a consequence of the Dimension Lemma \ref{Dimlemma}. Consequently (\ref{I.2}) has been proved.
\item\label{I.3} We show that $(\forall v\in {\mathbb E}^{c})\ o\notin \Phi(v)$. Indirect: Suppose that $o\in\Phi(v)$ and let $v\in (x,z)$. Since $\Phi$ is a homomorphism fulfilling (\ref{intersection-equals-empty-set}) we obtain by application of (\ref{I.1}) 
\begin{equation}\label{o-element-convexhull-x-y}
o\in \Phi(v) \subseteq \Phi(x) \vee \Phi(z)\subseteq g
\end{equation}
for some line $g\in {\mathbb E}^{c+1}$. Chose $y\in (x,v)$. Then
\begin{equation}\label{o-element-convexhull-x-y-2}
o\in \Phi(v) \subseteq \Phi(y) \vee \Phi(z)\subseteq g.
\end{equation}
Since we know from (\ref{intersection-equals-empty-set}) that 
\begin{equation*}
\Phi(x)\cap \Phi(v) = \Phi(y)\cap \Phi(v) = \Phi(z)\cap \Phi(v) = \emptyset
\end{equation*}
and since $\Phi(x), \Phi(y), \Phi(z)$ are convex, we obtain from (\ref{o-element-convexhull-x-y}) and (\ref{o-element-convexhull-x-y-2}) that there exists a ray $r\subseteq g$ emanating from $o$ such that $\Phi(x)$ and $\Phi(y)$ both intersect $r\setminus \{ o \}$ contradicting (\ref{I.2}). Thus (\ref{I.3}) has been shown.
\item\label{I.4} From Lemma \ref{radon-affine-dependence-lemma} (\ref{second-case}) we obtain that\footnote{Note that in case of Lemma \ref{radon-affine-dependence-lemma} (\ref{first-case}) the sets $\Phi(x)$ consist for any $x\in {\mathbb E}^{c}$ just of one single point, contradicting the hypotheses of the Lemma to be proved.} for any $c+2$ element set $S\subset {\mathbb E}^{c}$ there exists a hyperplane $F\subset {\mathbb E}^{c+1}$ with $o\notin F$ that intersects $\Phi(x)$ for any $x\in S$ i.e. $(\forall x\in S)\ \ \Phi(x)\cap F\neq \emptyset$.
\item\label{I.5} By application of the Transversality Theorem \ref{transversality theorem} of Appendix \ref{transversality-theorems} we obtain from (\ref{I.4}), (\ref{I.3}) and (\ref{I.1}) that there exists a hyperplane $F\subset {\mathbb E}^{c+1}$ with $o\notin F$ that intersects $\Phi(x)$ for any $x\in {\mathbb E}^{d}$ in a single point $p_{x}$. Since by (\ref{intersection-equals-empty-set}) $x\neq y$ implies $\Phi(x)\cap \Phi(y)=\emptyset$ and thus $p_{x}\neq p_{y}$ the function $\psi:{\mathbb E}^{c}\to F$ given by $\psi(x):=p_{x}$ is injective.
\item\label{I.6} By Proposition \ref{preservation-of-point-order-prop} the function $\psi:{\mathbb E}^{c}\to F$ defined in (\ref{I.5}) preserves the order of points and we obtain from Corollary \ref{corollary-aff-geometry} that $\psi$ is an affine bijection between ${\mathbb E}^{c}$ and $F$.
\item\label{I.7} Define the functions $\phi:{\mathbb E}^{c}\to {\mathbb E}^{c+1}$ and $\hat{\phi}: \mathbb{E}^{c} \to \mathbb{E}^{c+1}$ by letting $\phi (v)$, and 
${\hat {\phi }}(v)$ the farthest, and closest point of $\Phi (v)$ to $o$. 
Then we have 
\begin{equation} \label{Phi ist interval phi, phi dach}
\Phi(v) = [\phi(v), \hat{\phi}(v)].
\end{equation}
We show that the points $\phi(v)$ are located on a common hyperplane $H$ parallel to\footnote{Note that $H$ and $F$ may coincide.} $F$ i.e. we show that 
\begin{equation}\label{phi Dach in G}
\exists H \parallel F \hbox{ such that } H \hbox{ is a hyperplane and } (\forall v \in \mathbb{E}^{c})\ \phi(v) \in H.
\end{equation}
and thus---remind that $\psi:{\mathbb E}^{c}\to {\mathbb F}$ is an affine bijection and apply Proposition \ref{affine-bijection-prop}---that $\phi$ is an affine bijection between ${\mathbb E}^{c}$ and $H$.\\ \\
To this end, given a hyperplane $H \parallel F$, we say that $x \in \mathbb{E}^{c+1}$ lies above $H$ if $x$ is an element of the open half-space with boundary $H$ that does not contain $o$, while we say that $x$ lies below $H$ if $x$ is an element of the closed half-space with boundary $H$ that contains $o$. We proceed indirectly.\\ \\ 
Suppose that (\ref{phi Dach in G}) were not fulfilled. Then there exists a hyperplane $H \parallel F$ such that $\phi(x)$ lies above $H$ and $\phi(y)$ lies below $H$ for some points $x,y \in \mathbb{E}^{c+1}$. Either
\begin{equation} \label{case x,z above. y below}
\exists z \in \mathbb{E}^{c} \hbox{ with } z \notin \aff(x,y)\hbox{ such that } \phi(z) \hbox{ lies above }H
\end{equation}
or for all points $p \in \mathbb{E}^{c}$ with $p \notin \aff(x,y)$ we have that $\phi(p)$ lies below $H$. The second case is contradictory since for any such $p$ we have that $z:=2x-p \notin \aff(x,y)$ and $\phi(x) \in \Phi(p) \vee \Phi(z)$ and thus $\phi(z)$ lies above $H$, i.e. (\ref{case x,z above. y below}) holds.\\ \\
Thus let $z\notin \aff(x,y)$ be such that $\Phi(z)$ lies above $H$. Choose points $a,b \in \mathbb{E}^{c}$ such that 
\begin{equation} \label{y in affine hull}
y \in (a , x) \hbox{ and } y \in (b , z).
\end{equation}
Then $y=(a\vee x)\wedge (b\vee z)$ and since $\Phi$ is a homomorphism 
\begin{equation*}
\Phi(y)=(\Phi(a)\vee \Phi(x))\wedge (\Phi(b)\vee \Phi(z)).
\end{equation*} 
Herefrom we obtain by (\ref{Phi ist interval phi, phi dach}) and since $\psi(v)\in \Phi(v)$ that
\begin{equation} \label{eqn phi hat}
[\phi(y),\hat{\phi}(y)]=\Phi(y) \supseteq (\phi(x) \vee {\psi}(a)) \wedge (\phi(z) \vee {\psi}(b)).
\end{equation}
If we move $a$ and $b$ toward infinity---still fulfilling (\ref{y in affine hull})---we obtain from (\ref{eqn phi hat})---and since $\psi$ is by (\ref{I.6}) an affine bijection---that $\phi(y)$ lies on the same side of the hyperplane $H$ as $\phi(x)$ and $\phi(z)$, i.e. $\phi(y)$ lies above $H$. \emph{Contradiction.}\\ \\
Thus the existence of a hyperplane $H\parallel F$ such that (\ref{phi Dach in G}) holds has been proved and thus moreover $\phi:{\mathbb E}^{c}\to H$ is an affine bijection.
\item\label{I.8} One proves analogously to (\ref{I.7})---interchanging the words above and below and the functions $\phi$ and ${\hat{\phi}}$---that the function $\hat{\phi}$ implicitly defined by (\ref{Phi ist interval phi, phi dach}) is an affine bijection between ${\mathbb E}^{c}$ and some hyperplane $G\parallel F$. It is clear from the fact that $\Phi(x)$ is a proper line-segment for some $x\in {\mathbb E}^{c}$ and thus $\phi(x)\neq \hat{\phi}(x)$ that $H\neq G$. Further it is easily seen that $o$ is located on the same side of $H$ and $G$, and $o$ lies in the open half-space bounded by $G$ and not containing $F$ and thus especially $o\notin F\cup G\cup H$.
\item\label{I.9} By (\ref{I.8}), (\ref{phi Dach in G}), (\ref{Phi ist interval phi, phi dach}), 
the fact that $(\forall x \in \mathbb{E}^{c})\ \phi(x) \in H$ and Proposition \ref{ray hyperlane intersection existance gamma} we obtain that $\Phi(x) = [\phi(x),\gamma \phi(x) + (1-\gamma) o]$ for some $\gamma\in (0,1)$ and some affine function $\phi$, i.e. we obtain the first part of (\ref{segments-on-rays-result}). Hence $\bigcup_{x \in \mathbb{E}^{c}} \Phi(x)$ is convex and by application of Proposition \ref{Prop Bedingungen für Phi} and the fact that $\Phi(\emptyset)=\emptyset$ we obtain that $\Phi(C)=\bigcup_{x\in C} \Phi(x)\cup \Phi(\emptyset)=\bigcup_{x\in C} \Phi(x)$ i.e. we obtain the second part of (\ref{segments-on-rays-result}). Altogether we proved that (\ref{non-parallel-case}) implies case a) of the Lemma.
\end{enumerate}
\item Suppose now that (ii) holds i.e. there exist two distinct parallel lines in ${\cal G}$.
\begin{enumerate}[\hspace{-0.8em} (1)]
\item\label{II.1} We show that all lines in ${\cal G}$ are parallel. Let $x,y \in \mathbb{E}^{c}$ be distinct points such that the lines $g_{x}:=\aff(\Phi(x))$ and $g_{y}:=\aff(\Phi(y))$ are distinct and parallel. By Proposition \ref{more-than-one-line-segment-prop} there exists $z\in \mathbb{E}^{c}$ such that $\Phi(z)$ is a proper line-segment and
\begin{equation} \label{phi z nicht in aff phi x u phi y}
\Phi(z)\nsubseteq \aff(\Phi(x) \vee \Phi(y)).
\end{equation}
We let let $g_{z}:= \aff(\Phi(z))$ and proceed indirectly. Suppose that 
$g_{z} \nparallel g_{x}$ (and thus equivalently $g_{z} \nparallel g_{y}$).
Since by application of the Dimension Lemma \ref{Dimlemma} 
\begin{equation*} 
\dim (\Phi(x) \vee \Phi(z)) = \dim (\Phi(y) \vee \Phi(z)) \leq 2
\end{equation*}
we obtain 
\begin{equation*} 
g_{x} \cap g_{z} \neq \emptyset \hbox{ and } g_{y} \cap g_{z} \neq \emptyset
\end{equation*}
and hence $\Phi(z) \subseteq \aff(\Phi(x) \vee \Phi(y))$ contradicting (\ref{phi z nicht in aff phi x u phi y}).\\ \\
Thus $\Phi(z), \Phi(x)$ and $\Phi(y)$ are proper parallel distinct line-segments such that none of them is contained in the affine hull of the union of the other two. This implies that for any $w \in \mathbb{E}^{c}$ with $\Phi(w)$ a proper line-segment there exist two points $p,q \in \lbrace x,y,z \rbrace \subset \mathbb{E}^{c}$ such that $\Phi(p), \Phi(q)$ are parallel proper line-segments such that $\Phi(w) \nsubseteq \aff(\Phi(p) \vee \Phi(q))$. Repeating the argument from above we obtain that $\Phi(w)$ is parallel to the line-segments $\Phi(x), \Phi(y)$ and $\Phi(z)$. Hence all points that are mapped to proper line-segments are mapped to parallel line segments.
\item\label{II.2} That for any line $h$ parallel to $g:=g_{x}$ there does not exist more than one $v\in {\mathbb E}^{c}$ with $\Phi(v)\cap  g\neq \emptyset$ is established along the lines of (\ref{I.2}). One just has to replace the ray $r$ emanating from $o$ by a line $h$ parallel to $g$ and to use Proposition \ref{parallel-lines-two-dim-prop} instead of Proposition \ref{different-rays-two-dim-prop}.
\item\label{II.3} There is no need for an argument analogous to (\ref{I.3}) since parallel lines do not intersect in ${\mathbb E}^{d}$. Further the assertions (\ref{II.4}) to (\ref{II.6}) are porved analogous to (\ref{I.4}) to (\ref{I.9}) in the following way:
\item\label{II.4} From Lemma \ref{radon-affine-dependence-lemma-parallel} (\ref{second-case}) we obtain that\footnote{Note that in case of Lemma \ref{radon-affine-dependence-lemma-parallel} (\ref{first-case}) the sets $\Phi(x)$ consist for any $x\in {\mathbb E}^{c}$ just of one single point, contradicting the hypotheses of the Lemma to be proved.} for any $c+2$ element set $S\subset {\mathbb E}^{c}$ there exists a hyperplane $F\subset {\mathbb E}^{c+1}$ with $F\nparallel g$ that intersects $\Phi(x)$ for any $x\in S$ i.e. $(\forall x\in S)\ \ \Phi(x)\cap F\neq \emptyset$.
\item\label{II.5} By application of the Transversality Theorem \ref{parallel-transversality-theo} of Appendix \ref{transversality-theorems} we obtain from (\ref{II.4}) and (\ref{II.1}) that there exists a hyperplane $F\subset {\mathbb E}^{c+1}$ with $F\nparallel g$ that intersects $\Phi(x)$ for any $x\in {\mathbb E}^{d}$ in a single point $p_{x}$. Since by (\ref{intersection-equals-empty-set}) $x\neq y$ implies $\Phi(x)\cap \Phi(y)=\emptyset$ and thus $p_{x}\neq p_{y}$ the function $\psi:{\mathbb E}^{c}\to F$ given by $\psi(x):=p_{x}$ is injective.
\item\label{II.6} By Proposition \ref{preservation-of-point-order-prop-prallel} the function $\psi:{\mathbb E}^{c}\to F$ defined in (\ref{I.5}) preserves the order of points and we obtain from Corollary \ref{corollary-aff-geometry} that $\psi$ is an affine bijection between ${\mathbb E}^{c}$ and $F$.
\item\label{II.7} Let $w$ be a non-zero vector parallel to $g$ and let $\phi:{\mathbb E}^{c}\to {\mathbb E}^{c+1}$ and $\hat{\phi}: \mathbb{E}^{c} \to \mathbb{E}^{c+1}$ be such that (\ref{Phi ist interval phi, phi dach}) is fulfilled and the vector $\hat{\phi}(x)-\phi(x)$ points in the same direction as $w$ (provided that $\hat{\phi}(x)\neq \phi(x)$). Then we obtain in a way similar to (\ref{I.7})---using Proposition \ref{affine-bijection-prop-parallel} instead of Proposition \ref{affine-bijection-prop}---a hyperplanes $H\parallel F$ such that $\phi$ maps ${\mathbb E}^{c}$ to $H$.
\item\label{II.8} Analogous to (\ref{I.8}) we obtain a hyperpalen $G\neq H$ such that $G\parallel H$ and $\hat{\phi}$ maps ${\mathbb E}^{c}$ to $G$. 
\item\label{II.9} Finally we obtain---replacing Proposition \ref{ray hyperlane intersection existance gamma} by Proposition \ref{Prop phi+ = phi- + v} in an argument analogous to the one provided in (\ref{I.9})---that (\ref{case a}) implies case b) of the Lemma.
\end{enumerate}
\end{enumerate}
Thus the Lemma is proved.
\end{proof}

\section{Applications}\label{section-applications}

\begin{remark}
We investigate in this section anti-homomorphisms and homomorphisms between lattices of convex sets and spaces of convex functions. To this end we first introduce the space of convex lower semi-continuous\footnote{Note that a function $f:{\mathbb E}^{d}\to (-\infty,\infty]$ is lower semi-continuous iff its epigraph\linebreak $\epi(f):=\{ (x,r) \mid f(x)\leq r \}$ is closed in ${\mathbb E}^{d}\times (-\infty,\infty]$.} functions $Cvx(\mathbb{E}^{c})$ as well as its subspaces $Cvx_{(0,0)}(\mathbb{E}^{c})$, $Cvx_{[0,\infty]}(\mathbb{E}^{c})$, $Cvx_{(0,0)}^{\kappa}({\mathbb E}^{c})$ and $Cvx_{[0,\kappa]}(\mathbb{E}^{c})$.
\end{remark}

\begin{definition}
Let $f\in Cvx(\mathbb{E}^{c})$ if either $f: \mathbb{E}^{c} \to (-\infty,\infty]$ is convex and lower semi-continuous or $f \equiv -\infty$ i.e. let 
\begin{equation*}
Cvx(\mathbb{E}^{c}) := \{ f:{\mathbb E}^{c}\to (-\infty,\infty] \mid f \hbox{ is lower semi-continuous and convex} \}\cup \{ -\infty \}.
\end{equation*}
Let further 
\begin{equation*}
Cvx_{(0,0)}(\mathbb{E}^{c}) := \{ f\in Cvx(\mathbb{E}^{c}) \mid f(0)=0 \} \cup \{ -\infty \}
\end{equation*} 
and let
\begin{equation*}
Cvx_{[0,\infty]}(\mathbb{E}^{c}) := \{ f\in Cvx(\mathbb{E}^{c}) \mid f \geq 0\ \hbox{ and }\ (\exists x \in \mathbb{E}^{c})\ f(x)=0 \} \cup \{ +\infty \}.
\end{equation*}
\end{definition}

\begin{definition}\label{lattice-op-of-convex-functions-def}
Given two functions $f,g\in Cvx(\mathbb{E}^{c})$ we let
\begin{equation*}
\begin{split}
f\sqcap g:=\min \{ h\in Cvx(\mathbb{E}^{c}) \mid f, g \leq h \}\\[0.1cm]
f\sqcup g:=\max \{ h\in Cvx(\mathbb{E}^{c}) \mid f, g \geq h \}
\end{split}
\end{equation*}
with $\min$ and $\max$ denoting the minimum and maximum in $Cvx(\mathbb{E}^{c})$ with respect to point-wise order.\\ \\
Let further for functions $f,g\in Cvx_{(0,0)}(\mathbb{E}^{c})$
\begin{equation*}
\begin{split}
f \sqcap_{-} g:=\min \{ h\in Cvx_{(0,0)}(\mathbb{E}^{c}) \mid f, g \leq h \}\\[0.1cm]
f\sqcup_{-} g:=\max \{ h\in Cvx_{(0,0)}(\mathbb{E}^{c}) \mid f, g \geq h \}
\end{split}
\end{equation*}
with $\min$ and $\max$ denoting the minimum and maximum in $Cvx_{(0,0)}(\mathbb{E}^{c})$ with respect to point-wise order and let finally for $f,g\in Cvx_{[0,\infty]}(\mathbb{E}^{c})$
\begin{equation*}
\begin{split}
f\sqcap_{+} g:=\min \{ h\in Cvx_{[0,\infty]}(\mathbb{E}^{c}) \mid f, g \leq h \}\\[0.1cm]
f\sqcup_{+} g:=\max \{ h\in Cvx_{[0,\infty]}(\mathbb{E}^{c}) \mid f, g \geq h \}
\end{split}
\end{equation*}
with $\min$ and $\max$ denoting the minimum and maximum in $Cvx_{[0,\infty]}(\mathbb{E}^{c})$ with respect to point-wise order.
\end{definition}

\begin{remark}
Note that $f\sqcap g$, $f\sqcup g$, $f\sqcap_{-} g$, $f\sqcup_{-} g$, $f\sqcap_{+} g$ and $f\sqcup_{+} g$ are well defined i.e. $\max$ and $\min$ in Definition \ref{lattice-op-of-convex-functions-def} exist.\footnote{This is most easily seen by considering the epigraphs of the functions under consideration, since for example we may define $f\sqcup g$ as the uniquely determined function $h$ such that $\epi(h)$ equals the closed convex hull of $\epi(f)\cup \epi(g)$.}
\end{remark}

\begin{remark}\label{sqcups-relations}
Note further that $f,g\in Cvx_{(0,0)}(\mathbb{E}^{c})$ implies $f\sqcap_{-} g = f\sqcap g$, while
\begin{equation*}
f\sqcup_{-} g = f \sqcup g\ \hbox{ iff }\   [f\sqcup g](0)=0\ \hbox{ and }\ f\sqcup_{-} g = -\infty\ \hbox{ otherwise.}
\end{equation*}
Further $f,g\in Cvx_{[0,\infty]}(\mathbb{E}^{c})$ implies $f\sqcup_{+} g = f\sqcup g$, while 
\begin{equation*}
f\sqcap_{+} g = f\sqcap g\ \hbox{ iff }\ \min_{x} [f\sqcap g](x) = 0\ \hbox{ and }\ f\sqcap_{+} g = +\infty\ \hbox{ otherwise.}
\end{equation*}
\end{remark}

\begin{notation}
Given $C\subseteq {\mathbb E}^{c}$ we denote by $1_{C}^{\infty}$ the function $1_{C}^{\infty}:{\mathbb E}^{c}\to \{ 0 , \infty \}$ that is given by $1_{C}^{\infty}(x):=0$ for $x\in C$ and $1_{C}^{\infty}(x):=\infty$ for $x\notin C$.
\end{notation}

\begin{notation}
Given a compact, convex set $C\subseteq {\mathbb E}^{c}$ we let $h_{C}(y):=\sup_{x\in C} \langle x| y \rangle$ and call $h_{C}$ the support function of the convex set $C$. We further let
\begin{equation*}
{\cal S}({\mathbb E}^{c}):=\{ h_{C} \mid C\subseteq {\mathbb E}^{c}\ \hbox{convex and compact} \} 
\end{equation*}
i.e. ${\cal S}({\mathbb E}^{c})$ denotes the space of all support functions $h_{C}:{\mathbb E}^{c}\to {\mathbb R}$ of compact convex sets. 
\end{notation}

\begin{remark}
Note that $h_{\emptyset} = -\infty$.
\end{remark}

\begin{definition}
Define the Legendre-Fenchel transform as the function\newline \mbox{${\cal L}:Cvx({\mathbb E}^{c}) \to Cvx({\mathbb E}^{c})$} given by
\begin{equation*}
[{\cal L}f](y) := \sup_{x \in \mathbb{E}^{c}} (\langle x | y \rangle - f(x)).
\end{equation*}
\end{definition}

\begin{remark}\label{properties-of-fenchel-trafo}
It is a well known fact that ${\cal L}:Cvx({\mathbb E}^{c}) \to Cvx({\mathbb E}^{c})$ is an involution i.e. ${\cal L}\circ {\cal L} = id$ and thus a bijection. It is further a lattice-anti-endomorphism on $((Cvx({\mathbb E}^{c}), \sqcap,\sqcup )$ and thus order reversing i.e. we have that $f\leq g \Rightarrow {\cal L}f\geq {\cal L}g$ for $f,g\in Cvx({\mathbb E}^{c})$. Further given $g\in Cvx({\mathbb E}^{c})$, $C\in {\mathscr C}({\mathbb E}^{c})$ and $\kappa\in {\mathbb R}$ the Legendre-Fenchel transform fulfils
\begin{equation*}
h_{C}={\cal L}[1^{\infty}_{C}]\ \hbox{ and }\ {\cal L}[g+\kappa]={\cal L}[g] -\kappa.
\end{equation*}
For a derivation of these and further properties of the Legendre-Fenchel transform consult \cite[Section 11]{RW09}
\end{remark}

\begin{lemma}\label{legendre-trafo-is-bijection-lemma}
The Legendre-Fenchel transform ${\cal L}$ is a bijective lattice anti-homomorphism between $(Cvx_{(0,0)}({\mathbb E}^{c}),\sqcap_{-},\sqcup_{-})$ and $(Cvx_{[0,\infty]}({\mathbb E}^{c}),\sqcap_{+},\sqcup_{+})$.
\end{lemma}

\begin{proof}
Since we know from Remark \ref{properties-of-fenchel-trafo} that ${\cal L}:Cvx({\mathbb E}^{c})\to Cvx({\mathbb E}^{c})$ is a bijective lattice anti-endomorphism with respect to $\sqcap$ and $\sqcup$, it suffices by Remark \ref{sqcups-relations} to show that ${\cal L}$ is a bijection from $Cvx_{(0,0)}$ onto $Cvx_{[0,\infty]}$ that maps $\{ + \infty \}$ to $\{ -\infty \}$ and vice versa. Since ${\cal L}$ is an involution and ${\cal L}(+\infty)=-\infty$ it thus further suffices to show that
\begin{enumerate}[(i)]
\item\label{(i)} ${\cal L}(Cvx_{(0,0)}\setminus \{-\infty\})\subseteq Cvx_{[0,\infty]}\setminus \{+\infty\}$ and
\item\label{(ii)} ${\cal L}(Cvx_{[0,\infty]}\setminus \{+\infty\})\subseteq Cvx_{(0,0)}\setminus \{-\infty\}$.
\end{enumerate}
To prove (\ref{(i)}) let $f\in Cvx_{(0,0)}\setminus \{-\infty \}$. Since $f(0)=0$ there exists by convexity of $f$ and the theorem of Hahn-Banach some $\tilde{y}\in {\mathbb E}^{c}$ such that 
\begin{equation*}
(\forall x\in {\mathbb E}^{c})\ (\langle x , \tilde{y} \rangle \leq f(x))
\end{equation*}
and thus
\begin{equation}\label{Lf-attains-0}
[{\cal L}f](\tilde{y}) = \sup_{x} (\langle x | \tilde{y} \rangle - f(x)) = \langle 0 | \tilde{y} \rangle - f(0) = 0. 
\end{equation}
Further for any $y\in {\mathbb E}^{c}$ we have
\begin{equation}\label{Lf>=0}
[{\cal L}f](y) = \sup_{x} (\langle x | y \rangle - f(x)) \geq \langle 0 | y \rangle - f(0) = 0.
\end{equation}
Since we know that ${\cal L}f\in Cvx$ the conjunction of (\ref{Lf-attains-0}) and (\ref{Lf>=0}) just says that ${\cal L}f\in Cvx_{[0,\infty]}\setminus \{+\infty \}$ and (\ref{(i)}) has been shown.\\ \\
To prove (\ref{(ii)}) let $f\in Cvx_{[0,\infty]}\setminus \{ +\infty \}$ and let $x_{0}\in {\mathbb E}^{c}$ be such that $f(x_{0})=0$. Then
\begin{equation*}
[{\cal L}f](0) = \sup_{x} (\langle x | 0\rangle - f(x)) = 0 - f(x_{0})= 0 - 0 = 0
\end{equation*}
i.e. ${\cal L}f\in Cvx_{(0,0)}\setminus \{ -\infty \}$ and (\ref{(ii)}) has been shown.
\end{proof}

\begin{definition}
For $\kappa\in (0,\infty)$ we define
\begin{equation*}
\begin{split}
Cvx_{[0,\kappa]}({\mathbb E}^{c}):=\{ & f:{\mathbb E}^{c}\to [0,\kappa]\cup \{ + \infty \} \mid \\
& f\in Cvx_{[0,\infty]}({\mathbb E}^{c})\ \hbox{and}\ f^{-1}([0,\kappa])\ \hbox{is compact}\}.
\end{split}
\end{equation*}
\end{definition}

\begin{definition}
For $\kappa\in (0,\infty)$ we define
\begin{equation*}
Cvx_{(0,0)}^{\kappa}({\mathbb E}^{c}) := \{ g\in Cvx_{(0,0)}({\mathbb E}^{c}) \mid \exists h\in {\cal S}({\mathbb E}^{c})\ \hbox{such that}\ h-\kappa\leq g\leq h \}.
\end{equation*}
\end{definition}
\begin{lemma}
The Legendre transform ${\cal L}$ is for $\kappa\in (0,\infty)$ a bijective lattice anti-homomorphism between $( Cvx_{(0,0)}^{\kappa}({\mathbb E}^{c}),\sqcap_{-},\sqcup_{-} )$ and $( Cvx_{[0,\kappa]}({\mathbb E}^{c}),\sqcap_{+},\sqcup_{+} )$.
\end{lemma}

\begin{proof}
From Lemma \ref{legendre-trafo-is-bijection-lemma} we already know that ${\cal L}$ is a bijective lattice anti-homomorphism between $(Cvx_{(0,0)}({\mathbb E}^{c}),\sqcap_{-},\sqcup_{-})$ and $(Cvx_{[0,\infty]}({\mathbb E}^{c}),\sqcap_{+},\sqcup_{+})$.\\ \\
Thus---since ${\cal L}$ is 
an involution---it suffices to show that
\begin{enumerate}[(i)]
\item\label{(i)*} ${\cal L}(Cvx_{[0,\kappa]})\subseteq Cvx_{(0,0)}^{\kappa}$ and
\item\label{(ii)*} ${\cal L}(Cvx_{(0,0)}^{\kappa})\subseteq Cvx_{[0,\kappa]}$.
\end{enumerate}
To prove (\ref{(i)*}) it thus suffices to show that $f\in Cvx_{[0,\kappa]}$ implies that there exists some convex body $C$ such that
\begin{equation}\label{squeezed-between-support-functions}
h_{C}-\kappa \leq {\cal L}[f] \leq h_{C}.
\end{equation}
Given $f\in Cvx_{[0,\kappa]}$ we let $C:=f^{-1}([0,\kappa])$ and obtain 
\begin{equation}\label{bounding-f-equation}
1^{\infty}_{C}\leq f\leq 1^{\infty}_{C}+\kappa.
\end{equation}
By application of Remark \ref{properties-of-fenchel-trafo} we obtain from (\ref{bounding-f-equation}) that
\begin{equation*}
h_{C}-\kappa = {\cal L}[1^{\infty}_{C}]-\kappa = {\cal L}[1^{\infty}_{C}+\kappa] \leq {\cal L}[f] \leq {\cal L}[1^{\infty}_{C}]=h_{C}
\end{equation*}
i.e. (\ref{squeezed-between-support-functions}) and thus (\ref{(i)*}) has been proved.\\ \\
To prove (\ref{(ii)*}) it suffices by Lemma \ref{legendre-trafo-is-bijection-lemma} to show for $g\in Cvx_{(0,0)}^{\kappa}$ that 
\begin{equation*}
h_C - \kappa \leq g \leq h_C\ \hbox{ implies }\ 1^{\infty}_{C} \leq \mathcal{L}(g) \leq 1^{\infty}_{C} + \kappa.
\end{equation*}
This is again done by application of Remark \ref{properties-of-fenchel-trafo}.
\end{proof}

\begin{theorem} \label{Homomorphismen Satz Anwendungen}
Let $\kappa\in (0,\infty)$, let $c\geq 3$ and let 
\begin{equation*}
{\mathcal H}:({\mathscr C}({\mathbb E}^{c}),\wedge,\vee) \to (Cvx_{[0,\kappa]}(\mathbb{E}^{c}),\sqcap_{+},\sqcup_{+})
\end{equation*}
be a non-trivial lattice homomorphism. Then ${\mathcal H}(C)={1}^{\infty}_{\phi(C)}$ for some affine bijection $\phi:{\mathbb E}^{c}\to {\mathbb E}^{c}$.
\end{theorem}

\begin{proof}
It is easily seen that the transformation 
\begin{equation}
{\cal T}:Cvx_{[0,\kappa]}(\mathbb{E}^{c})\to {\mathscr C}(\mathbb{E}^{c}\times {\mathbb R})\ \hbox{ given by }\ f\mapsto epi(f)\cap ({\mathbb E}^{c}\times [0,\kappa])
\end{equation}
is a non-trivial bijective lattice homomorphism. Thus 
\begin{equation*}
\Phi:={\cal T}\circ {\cal H}:{\mathscr C}({\mathbb E}^{c})\to {\mathscr C}(\mathbb{E}^{c}\times{\mathbb R})
\end{equation*} is a non-trivial lattice homomorphism. By Theorem \ref{Theorem Homomorphisms} we have four possible cases regarding the representation of $\Phi$. We are going to show that only the case (\ref{case iii}) of parallel line-segments is relevant for our situation.\\ \\
By Theorem \ref{Theorem Homomorphisms} we obtain for any $\xi\in {\mathbb E}^{c}$ that $\Phi(\xi)=[a_{\xi},b_{\xi}]$ with $a_\xi, b_\xi \in \mathbb{E}^{c+1}$.  
Thus ${\cal H}(\xi)={\cal T}^{-1}([a_{\xi},b_{\xi}])=f$ with $f \in Cvx_{[0,\kappa]}(\mathbb{E}^{c})$ such that 
\begin{equation} \label{epi cap}
epi(f)\cap ({\mathbb E}^{c}\times [0,\kappa])=[a_{\xi},b_{\xi}].
\end{equation}
Since there exists some $x_{\xi}$ such that $f(x_{\xi})=0$ and $\kappa > 0$, we obtain from (\ref{epi cap}) that $a_{\xi}=(x_{\xi},0)$ and $b_{\xi}=(x_{\xi},\kappa)$ and thus $[a_{\xi}, b_{\xi}]\parallel {\mathbb R} \times \lbrace 0 \rbrace$ i.e. we are in the case of parallel line-segments.\\ \\
Thus if we let $\phi(\xi)=x_{\xi}$ the mapping $\phi:{\mathbb E}^{c}\to {\mathbb E}^{c}$ is according to Theorem \ref{Theorem Homomorphisms}, case (\ref{case iii}), a bijective affinity and $\Phi(C)=\phi(C)+[(0,0),(0,\kappa)]$. Thus ${\cal H}(C)={\cal T}^{-1}\circ \Phi(C)=1^{\infty}_{\phi(C)}$ and the theorem is proved.
\end{proof}

\begin{theorem}\label{anti-homo-theorem}
Let $\kappa\in (0,\infty)$, let $c\geq 3$ and let 
\begin{equation*}
{\Lambda}:({\mathscr C}({\mathbb E}^{c}),\wedge,\vee) \to (Cvx_{(0,0)}^{\kappa}(\mathbb{E}^{c}),\sqcap_{-},\sqcup_{-})
\end{equation*}
be a non-trivial lattice anti-homomorphism. Then there exists an affine bijection $\phi:{\mathbb E}^{c}\to {\mathbb E}^{c}$ such that ${\Lambda}(C)={\mathcal L}[1^{\infty}_{\phi(C)}]$ with ${\mathcal L}$ denoting the Legendre-Fenchel transformation.
\end{theorem}

\begin{proof}
Since the Legendre-Fenchel transform ${\cal L} : Cvx_{(0,0)}^{\kappa}(\mathbb{E}^{c})\to Cvx_{[0,\kappa]}(\mathbb{E}^{c})$ is a non-trivial anti-homomorphism we obtain that
\begin{equation*}
{\cal H}:={\cal L}\circ {\Lambda}:{\mathscr C}({\mathbb E}^{c})\to Cvx_{[0,\kappa]}(\mathbb{E}^{c})\hbox{ is a non-trivial homomorphism.}
\end{equation*}
Thus for some affine bijection $\phi:{\mathbb E}^{c}\to {\mathbb E}^{c}$ we have in accordance with Theorem \ref{Homomorphismen Satz Anwendungen} that ${\mathcal H}(C)={1}_{\phi(C)}^{\infty}$  and thus---since ${\cal L}$ is by Remark \ref{properties-of-fenchel-trafo} an involution---we obtain that ${\Lambda}={\cal L}\circ {\cal H}={\cal L}[{1}^{\infty}_{\phi(C)}]$.  
\end{proof}

\begin{appendix}

\section{Nontrivial Endomorphisms}\label{endomorphisms}

\begin{remark}
We display below the main result of \cite{Gr91}. We outline a short proof of the result based on Proposition \ref{injectivity-on-points-prop}, the Dimension Lemma \ref{Dimlemma} and the proof of Lemma \ref{Gruberfalllemma}.
\end{remark}

\begin{theorem} \label{Grubersatz}
For $c\geq 2$ a mapping $\Phi : \mathscr{C}({\mathbb E}^{c}) \to \mathscr{C}({\mathbb E}^{c})$ is a non trivial endomorphism of the lattice $( \mathscr{C}({\mathbb E}^{c},\wedge,\vee )$ if and only if there exists an affine bijection $ \phi : \mathbb{E}^{c} \to \mathbb{E}^{c} $ such that $ \Phi(C) = \phi(C)$ for each $ C \in \mathscr{C}(\mathbb{E}^{c})$.
\end{theorem}

\begin{proof}
From the Dimension Lemma \ref{Dimlemma} we obtain that the $(-1)$ dimensional empty set has to be mapped to the empty set. From Proposition  \ref{injectivity-on-points-prop} we obtain that the $0$-dimensional one-point sets can not be mapped to the empty set and thus by the Dimension Lemma \ref{Dimlemma} have to be mapped to one-point sets. The proof is concluded by the very same argument that proves Lemma \ref{Gruberfalllemma}.
\end{proof}

\section{Transversality Theorems}\label{transversality-theorems}

\begin{remark}
In this section we provide two transversality theorems, one for each case considered in the proof of Lemma \ref{empty-to-empty-lemma}. The transversality theorem for the case of parallel line-segments, Theorem \ref{parallel-transversality-theo}, is a well known consequence of a theorem of S. Karlin and L.S. Shapley \cite{KarlinSharpley1950} and thus of Helly's theorem (see Danzer, Grünbaum and Klee \cite[Remarks following Section 2]{DanzerGruenbaumKlee}). The proof of its two dimensional version \cite[Theorem 6.8]{Valentine}, \cite[Section 2.2]{DanzerGruenbaumKlee} 
extends without any difficulty to the general case. For the readers convenience we provide a proof-sketch in Remark \ref{proof-transversal-remark}. Our transversality result for line-segments directed to a given point $o\in {\mathbb E}^{d}$, Theorem \ref{transversality theorem}, is a seemingly less trivial consequence of Helly's theorem. 
\end{remark}

\begin{theorem}[Transversality theorem for parallel segments]\label{parallel-transversality-theo}
Let ${\cal I}$ be a family of possibly degenerate compact line-segments $I\subseteq {\mathbb E}^{d}$ such that $I\in {\cal I}$ implies that $I\subseteq g_{I}$ with $g_{I}$ some line parallel to a given line $g\subset {\mathbb E}^{d}$. If for any $(d+1)$-element set ${\cal K}\subseteq {\cal I}$ there exists a hyperplane $H_{\cal K}$ not parallel to $g$ that intersects any line-segment $I\in {\cal K}$, then there exists a hyperplane $H$ not parallel to $g$ that intersects any line-segment $I\in {\cal I}$ (in a single point $p_{I}$).
\end{theorem}

\begin{remark}\label{proof-transversal-remark}
To prove Theorem \ref{parallel-transversality-theo} suppose without loss of generality that $g\parallel \{ 0 \} \times {\mathbb R} \subseteq {\mathbb E}^{d-1}\times {\mathbb R} = {\mathbb E}^{d}$. Thus a hyperplane $H$ is not parallel to $g$ iff $H$ is the graph of a linear functional $\psi_{H}:{\mathbb E}^{d-1}\to {\mathbb R}$. Further any compact line-segment $I\parallel g$ is of the form $I=[ (x_{I}, a_{I}), (x_{I}, b_{I})]$ with $x_{I}\in {\mathbb E}^{d-1}$ and $a_{I},b_{I}\in {\mathbb R}$. Consequently $H$ intersects the compact line-segment $I$ iff $\psi_{H}(x_{I})\in [a_{I},b_{I}]$. Since $\{ \psi \mid \psi(x_{i})\in [a_{I},b_{I}] \}$ is convex Theorem \ref{parallel-transversality-theo} follows after application of a compactness argument from Helly's theorem.   
\end{remark}

\begin{definition}\label{cal H-iota-notation}
Let $o\in {\mathbb E}^{d}$ be fixed and let ${\cal H}$ be the space of all hyperplanes $H\subseteq {\mathbb E}^{d}$ with $o\notin H$. Denote by $n_{H}$ the unit normal vector onto $H\in {\cal H}$ pointing to the opposite side of $o$ and let $\Delta_{H}$ denote the (minimal) distance of $H$ to $o$. Define further a function $\iota:{\cal H}\to {\mathbb E}^{d}$ by $\iota(H):=n_{H}/\Delta_{H}$.
\end{definition}

\begin{remark}\label{iota-def+prop}
The function $\iota$ introduced in Definition \ref{cal H-iota-notation} maps ${\cal H}$ bijectively onto ${\mathbb E}^{d}\setminus \{ 0 \}$. We call any $H\in {\cal H}$ a polar with corresponding pole $\iota(H)$ with respect to the unit circle around $o$ (cf. \cite[Definition 11.3.13]{Zieschang}).
\end{remark}

\begin{notation}
Let $o\in {\mathbb E}^{d}$ be fixed and let ${\cal M}$ be a family of sets $M\subset {\mathbb E}^{d}$. Then we let ${\cal H}_{\cal M}:=\{ H\in {\cal H} \mid (\forall M\in {\cal M})\ M\cap H\neq \emptyset \}$.  
\end{notation}

\begin{lemma}\label{convexity-of-family-of-hyperplanes}
Let $\iota$ denote the function introduced in Proposition \ref{iota-def+prop}. Let $r$ be a ray emanating from $o\in {\mathbb E}^{d}$ and let $C\subset r\setminus \{ o \}$ be convex i.e. let $C$ be a line-segment on $r\setminus \{ o \}$. Then 
\begin{equation*}
\iota( {\cal H}_{\{ C\} })\ \hbox{ is convex}.
\end{equation*}
If $C$ is additionally compact i.e. a compact line-segment then $\iota( {\cal H}_{\{ C\} } )$ is a closed convex set.
\end{lemma}

\begin{remark}
If we identify maps with their graphs, it is well known that the mapping $\Pol := \iota \cup \iota^{-1}$ fulfils $x\in H \Longleftrightarrow \Pol(H) \in \Pol(x)$ (compare with \cite[Theorem 11.3.14]{Zieschang} in the two dimensional case). This fact can be used to simplify the following argument. 
\end{remark}

\begin{proof}[Proof of Lemma \ref{convexity-of-family-of-hyperplanes}]
We just prove the convexity of $\iota( {\cal H}_{\{ C\} })$. Let $\underline{H}, \overline{H}\in  {\cal H}_{\{ C\}}$. We have to show that 
\begin{equation}\label{convexity-of-H_C}
[\iota(\underline{H}),\iota(\overline{H})]\subseteq \iota( {\cal H}_{\{ C\}} ).
\end{equation}
We distinguish the cases that $\underline{H}\parallel \overline{H}$ and $\underline{H}\nparallel \overline{H}$. The first case is almost trivial, thus we consider the second one. Let $G=\underline{H}\cap \overline{H}$. Note that $G$ is a $(d-2)$-dimensional subspace of ${\mathbb E}^{d}$. Let ${\cal H}_{\supset G}$ be the family of all $H\in {\cal H}$ with $G\subset H$. Let $F$ be the plane through $o$ perpendicular to $G$ and suppose without loss of generality that $0\in F$ i.e. that $F$ is a vector space.\\ \\
Denote by $p_{G}$ the unique intersection-point of $F$ and $G$ and let $m=\frac{1}{2}(p_{G}+o)$. Note that the foot-point $p_{H}$ of $o$ in the hyperplane $H$ is for any $H\in {\cal H}_{\supset G}$ an element of $F$. Thus 
$\iota(H)$ is for any $H\in {\cal H}_{G}$ an element of $F$ and we are able to argue entirely in $F$.\\ \\
By a theorem of Thales we obtain that the points $p_{H}$ are for $H\in {\cal H}_{\supset G}$ all located on the circle $R\subseteq F$ with center $m$ and radius $\Vert o-m\Vert_{2}$. We can thus suppose that the elements of ${\cal H}_{\supset G}$ are parametrized by the angle $\alpha\in (0,2\pi)$ enclosed by $p_{H}-m$ and $o-m$ i.e. we let ${\cal H}_{\supset G} = \{ H_{\alpha} \mid \alpha\in (0,2\pi) \}$. We further identify without loss of generality $F$ with ${\mathbb E}^{2}$ and suppose further---by application of the appropriate orthogonal transformation followed by a dilation---that $m=(0,0)$, $o=(1,0)$ and thus $p_{G}=(-1,0)$. The coordinates of $p_{H_{\alpha}}$ are then given by $(\cos(\alpha),\sin(\alpha))$ and $\Delta_{H}^{2} = \Vert p_{H_{\alpha}} - o\Vert_{2}^{2} = (1-\cos(\alpha))^{2} + \sin^{2}(\alpha)= 2(1 - \cos(\alpha))$. Thus 
\begin{equation*}
\iota(H_{\alpha}) = n_{H}/\Delta_{H} = \frac{n_{H}\cdot \Delta_{H}}{\Delta_{H}^{2}} = \frac{1}{2} \left (-1,\frac{\sin(\alpha)}{1-\cos(\alpha)} \right ) = \frac{1}{2} \cdot (-1,\cot(\alpha/2)) 
\end{equation*}
This shows that $\alpha \mapsto \iota(H_{\alpha})$ is a strictly monotone and continuous mapping from ${\mathbb R}$ to a line. Thus we obtain for $H_{\underline{\alpha}}:=\underline{H}$ and $H_{\overline{\alpha}}:=\overline{H}$ that 
\begin{equation*}
[\iota(\underline{H}),\iota(\overline{H})] = \{ \iota(H_{\alpha}) \mid \alpha\in [\underline{\alpha},\overline{\alpha}] \}.
\end{equation*}
Since further for any $\alpha \in [\underline{\alpha},\overline{\alpha}]$ we have that $C\cap H_{\alpha}\neq \emptyset$, we obtain (\ref{convexity-of-H_C}).
\end{proof}

\begin{theorem}[Transversality theorem for segments directed towards $o$]\label{transversality theorem}
Let $o\in {\mathbb E}^{d}$ and let ${\cal I}$ be a family of possibly degenerate compact line-segments $I\subseteq {\mathbb E}^{d}$ such that $I\in {\cal I}$ implies that $I\subseteq r_{I}\setminus \{ o \}$ with $r_{I}$ some ray emanating from $o$. If for any $d+1$-element set ${\cal K}\subseteq {\cal I}$ we have that ${\cal H}_{\cal K}\neq \emptyset$ then ${\cal H}_{\cal I}\neq \emptyset$. I.e. if for any $d+1$-element set ${\cal K}\subseteq {\cal I}$ there exists a hyperplane $H\in {\cal H}$ that intersects any line-segment $I\in {\cal K}$, then there exists a hyperplane $H\in {\cal H}$ that intersects any line-segment $I\in {\cal I}$ (in a single point $p_{I}$).
\end{theorem}


\begin{proof}
Without loss of generality we prove the result just in the case that there exists some $d+1$-element set ${\cal K}_{0}$ such that $\aff(o,\bigcup {\cal K}_{0})={\mathbb E}^{d}$. 
(In case that this hypothesis does not hold just replace ${\mathbb E}^{d}$ by the unique maximal subspace for that the hypothesis is fulfilled.) 
By Lemma \ref{convexity-of-family-of-hyperplanes} the set $\iota ({\cal H}_{I})$ is for any $I\in {\cal I}$ closed and convex. Thus the hypotheses of the theorem say that the family $\{ \iota ({\cal H}_{I}) \mid I\in {\cal I} \}$ consists of closed convex non-empty subsets of ${\mathbb E}^{d}$ such that any $d+1$ of them possess non empty intersection. Further it is not difficult to see that $\iota ({\cal H}_{{\cal K}_{0}})$ is compact. Thus by application of a version of Helly's theorem (Theorem \ref{helly's} below) 
we obtain that $\iota ({\cal H}_{I})\neq \emptyset$ and thus ${\cal H}_{I}\neq\emptyset$.
\end{proof}

\begin{theorem}[Helly's Theorem, compare with \cite{DanzerGruenbaumKlee}, \cite{Gr07} and \cite{Valentine}]\label{helly's}
Let ${\cal C}$ be a family of closed convex sets in ${\mathbb E}^{d}$ such that for any $d+1$-element set ${\cal C}_{0}\subseteq {\cal C}$ we have $\bigcap {\cal C}_{0}\neq \emptyset$. Suppose further that for one $d+1$-element set ${\cal C}_{0}\subseteq {\cal C}$ we have that $\bigcap {\cal C}_{0}$ is compact. Then $\bigcap {\cal C} \neq \emptyset$.
\end{theorem}
$ $\\
{\bf Aknowledgment} We would like to thank Endre Makai Junior for reading main parts of the article and providing us with many extremely accurate observations (remarks, corrections and suggestions for improvement).

\end{appendix}


\begin{thebibliography}{99}





\bibitem{AM07.1}
S. Artstein-Avidan, V. Milman
\newblock{\em A chacaterization of the concept of duality}
\newblock{Electronic Research Announcements in Mathematical Sciences, Volume 14, Pages 42-59 (2007)}

\bibitem{AM09}
S. Artstein-Avidan, V. Milman
\newblock{\em The concept of duality in convex analysis, and the characterization of the Legendre transform}
\newblock{Annals of Mathematics, 169 (2009), 661-674}



\bibitem{AM09.1}
S. Artstein-Avidan, V. Milman
\newblock{\em Hidden structures in the class of convex functions and a new duality transform}
\newblock{J. Eur. Math. Soc. 13, 975-1004}

\bibitem{AM10}
S. Artstein-Avidan, V. Milman
\newblock{\em A characterization of the support map}
\newblock{Advances in Mathematics 223(2010), 379-391}

\bibitem{AMF11}
S. Artstein-Avidan, D. Florentin,  V. Milman  
\newblock{\em Order isomorphisms in windows}
\newblock{Electronic Research Announcements in Mathematical Sciences, Volume 18, Pages 112-118 (2011)}

\bibitem{ASlo11}
S. Artstein-Avidan, B.A. Slomka 
\newblock{\em Order isomorphisms in cones and a characterization of duality for ellipsoids}
\newblock{Sel. Math. New Ser. (2012) 18:391-415}

\bibitem{BS}
K.J. Böröczky, R. Schneider
\newblock{\em A characterization of the duality mapping for convex bodies}
\newblock{Geometric and Functional Analysis. 08/2008; 18(3):657-667}

\bibitem{DanzerGruenbaumKlee}
L. Danzer, B. Grünbaum, V. Klee
\newblock{\em Helly’s theorem and its relatives.}
\newblock{Proc. Sympos. Pure Math. 7, 101-180 (1963).}


\bibitem{Gr91}
P.M. Gruber
\newblock{\em The Endomorphisms of Convex Bodies}
\newblock{ Abh. Math. Sem. Univ. Hamburg, {\bf 61} (1991), 121-130}

\bibitem{Gr07}
P.M. Gruber
\newblock{\em Convex and Discrete Geometry.}
\newblock{Grundlehren der Mathematischen Wissenschaften 336. Berlin: Springer. (2007)}

\bibitem{KarlinShapley1950}
Karlin, L.S. Shapley
\newblock{\em Some applications of a theorem on convex functions}
\newblock{Ann. of Math. (2) {\bf 52} (1950), 148-153}  

\bibitem{Le58}
H. Lenz
\newblock{\em Einige Anwendungen der projektiven Geometrie auf
Flächentheorie}
\newblock{Math. Nachr. 18 (1958), 346-359.}


\bibitem{MSeSlo11}
V. Milman, A. Segal, B.A. Slomka
\newblock{\em A characterization of duality trough section/projection correspondence in the finite dimensional setting}
\newblock{Journal of Functional Analysis 261 (2011) 3366-3389}

\bibitem{RW09}
R. T. Rockafellar, R. J-B. Wets,
\newblock{\em Variational Analysis,}
\newblock{Grundlehren der Mathematischen Wissenschaften, Vol. 317, Springer, 1998 (3rd printing 2009).}

\bibitem{SeSlo}
A. Segal, B. Slomka
\newblock{\em Duality on convex sets in generalized regions}
\newblock{
Asymptotic Geometric Analysis,
Fields Institute Communications, Volume 68, 2013, pp 289-298 }

\bibitem{SLO11}
B. A. Slomka
\newblock{\em On duality and endomorphisms of lattices of closed convex sets}
\newblock{Adv. Geom. 11 (2011), 225-239}

\bibitem{Valentine}
F. A. Valentine
\newblock{\em Convex sets}
\newblock{McGraw-Hill Book Company, New York (1964)}

\bibitem{We2001}
H. Weisshaupt
\newblock{\em Isometries with respect to symmetric difference metrics}
\newblock{Studia scientiarum mathematicarum Hungarica, 37(3-4), 2001, pp. 273-318 }

\bibitem{Zieschang}
H. Zieschang
\newblock{\em Lineare Algebra und Geometrie}
\newblock{\em B.G. Teubner Stuttgart (1997).}




\end{thebibliography}
\end{document}